\documentclass[11pt]{article}
\usepackage{amsmath}
\usepackage{mathrsfs}
\usepackage{amsfonts}
\usepackage{amsmath,amsthm}
\usepackage{amsmath,amssymb,amsthm,latexsym}
\usepackage{graphics}
\usepackage{subfigure}
\usepackage{float,caption}
\usepackage{graphicx,arcs}
\usepackage{amscd}
\usepackage[all]{xy}
\usepackage{hyperref}

\newtheorem{theorem}{Theorem}[section]
\newtheorem{proposition}[theorem]{Proposition}

\newtheorem{definition}[theorem]{Definition}
\newtheorem{corollary}[theorem]{Corollary}
\newtheorem{lemma}[theorem]{Lemma}

\theoremstyle{remark}
\newtheorem{remark}[theorem]{Remark}

\def\<{\langle}
\def\>{\rangle}

\def\R{\mathbb{R}}

\begin{document}
\title{\bf{Can One Perturb the Equatorial Zone on a Sphere with Larger Mean Curvature? }}

\author{Baichuan Hu%
  \thanks{School of Mathematical Sciences, Peking University,
 Beijing 100871, People's Republic of China. \texttt{1900010607@pku.edu.cn}. }
  \and
Xiang Ma%
  \thanks{LMAM, School of Mathematical Sciences, Peking University,
 Beijing 100871, P.R. China. \texttt{maxiang@math.pku.edu.cn}, Fax:+86-010-62751801. Corresponding author. Supported by NSFC grant 11831005.}
\and
 Shengyang Wang%
  \thanks{School of Mathematical Sciences, Peking University,
 Beijing 100871, People's Republic of China. \texttt{1900010752@pku.edu.cn}. Supported by the Undergraduate Student Research Study program of Peking University for the year 2021-2022.} }

\date{\today}
\maketitle

\begin{center}
{\bf Abstract}
\end{center}

\hspace{2mm}
We consider the mean curvature rigidity problem of an equatorial zone on a sphere which is symmetric about the equator with width $2w$. There are two different notions on rigidity, i.e. strong rigidity and local rigidity. We prove that for each kind of these rigidity problems, there exists a critical value such that the rigidity holds true if, and only if, the zone width is smaller than that value.
For the rigidity part, we used the tangency principle and a specific lemma (the trap-slice lemma we established before). For the non-rigidity part, we construct the nontrivial perturbations by a gluing procedure called the round-corner lemma using the Delaunay surfaces.\\

{\bf Keywords:}  spheres, mean curvature, rigidity theorem, infinitesimal deformation, Delaunay surfaces, tangency principle, gluing construction.  \\

{\bf MSC(2010):\hspace{2mm} 53C24, 53C42; see also 52C25.}

\section{Introduction}

The central theme in this paper is about the so-called \emph{mean curvature rigidity} phenomenon.
The first result along this direction is by Gromov \cite{Gromov}, who pointed out that a hyperplane $M$ in a Euclidean space $\mathbb{R}^{n+1}$ cannot be perturbed on a compact set $S$ so
that the perturbed hypersurface $\Sigma$ has mean curvature $H_{\Sigma} \ge 0$ unless $H_{\Sigma}\equiv 0$ and $\Sigma=M$ identically. Very soon Souam \cite{Souam} gave a simple proof of this fact using the \emph{Tangency Principle} and established rigidity results for horospheres, hyperspheres, and hyperplanes in the hyperbolic space $\mathbb{H}^{n+1}$. For other types of rigidity theorems on spheres and hemispheres, we just mention the famous Min-Oo's conjecture \cite{MinOo} and a series of beautiful work \cite{Brendle, HangFB-06}.

In a previous work \cite{ChenMaWang}, we found that similar mean curvature rigidity result holds
for compact CMC hypersurfaces like spheres, with the restriction that the perturbed part is
no more than a hemisphere.
In other words, we considered perturbation of a spherical cap whose boundary is fixed up to $C^2$.
Here we turn to perturbations on a doubly connected domain of a sphere, and our aim is to find out when such rigidity theorem still holds true.
This is done by detailed analysis and comparison with the Delaunay CMC surfaces
and gluing constructions in the 3-dim space.
Yet this should be not difficult to generalize to any n-dimensional space.

Generally, suppose $N^{n+1}$ is a Riemannian manifold, and $M^n$ is an embedded hypersurface in it. The second fundamental form of $M^n$ and the mean curvature $\widetilde{H}(x), x \in M^n$ are defined as usual with respect to a given normal unit vector field $\widetilde{\pmb{n}} \in \Gamma (T^{\perp} M)$.
$S \subset M^n$ is a precompact open domain on $M^n$.

\begin{definition}
For $k \ge 2$, a $C^k$-perturbation of $S\subset M^n $ refers to another $C^k$ embedding $\Sigma:M^n \rightarrow N^{n+1}$ with $\Sigma = id$ in $N^{n+1} \backslash S$. When $k = \infty$, we say the perturbation is smooth.
\end{definition}

When there is no confusion, we will also use $\Sigma$ to represent $\Sigma(M^n)$. And we will only talk about the smooth perturbation (which is easy to generalize to other $C^k$-perturbations.\\

There also exists a unit normal vector field $\pmb{n} \in \Gamma (T^{\perp} \Sigma)$ with $\pmb{n} = \widetilde{\pmb{n}}$ in $M^n \backslash S$, which gives the mean curvature of $\Sigma$, defined as $H(x) \triangleq H(\Sigma(x)), x\in M$. Given a constant $\alpha \in \R$, we say $H(\Sigma) \ge \alpha$ iff $\forall x \in S$, $H(x) \ge \alpha$ (similar for $H(\Sigma) \le \alpha$).

In convex geometry and isometric deformation problems, usually we talk about two kinds of notions about deformations and rigidity.
One is the so-called \emph{infinitesimal deformations} which exist in an arbitrarily small neighborhood of the original hypersurface; one can imagine that it comes from a one-parameter deformation process. The other is \emph{large-scale} perturbations which have to go far away.
Here we need also to distinguish between these two kinds of rigidity.

\begin{definition}
Given an open domain $\Theta \subset N^{n+1}$ satisfying $\overline{S} \subset \Theta$. We say that $S$ has $H^+$(or $H^-$) rigidity in $\Theta$ if for any perturbation $\Sigma$ with $\Sigma(S) \subset \Theta$, the two statements below are equivalent:\\
(1)$H(x) \ge \widetilde{H}(x)$(or $H(x) \le \widetilde{H}(x)$), $\forall x \in S$ \\
(2)$\Sigma = id$ in $S$

We say $S$ has local $H^+$(or $H^-$) rigidity, if $\exists \Theta \subset N^{n+1}$ satisfying $\overline{S} \subset \Theta$, and $S$ has $H^+$(or $H^-$) rigidity in $\Theta$.

When $\Theta = N^{n+1}$, we simply say $S$ has (strong) $H^+$(or $H^-$) rigidity.

We can simply say equivalently that $S$ is (local/strong) $H^+/H^-$ rigid.
\end{definition}

\begin{remark}\label{monotonicity}
It is obvious that each of these four kinds of rigidity has monotonicity property with respect to the domain $S$, i.e. for two precompact open domain $S_1,S_2 \subset M^{n}$ with $S_1 \supset S_2$, we have:\\
(1) If $S_1$ has $H^+$ or $H^-$ rigidity, then this is also true for $S_2$.\\
(2) If $S_1$ has local $H^+$ or $H^-$ rigidity, then $S_2$ also has this rigidity property.

We can find that the \emph{strong rigidity} considers a large-scale perturbation of $S$. However, for the local $H$ rigidity, we only need to consider a local deformation of $S$, since we only have to prove the existence of some $\Theta$ which can be arbitrarily small enough. Notice that the strong rigidity implies the local rigidity.
\end{remark}

As a demonstration of these rigidity notions, we review and summarize our previous results as below:

\begin{theorem}\cite{ChenMaWang}
A spherical cap $S\subset S^n$ is $H^+$ rigid if and only if it is part of a hemisphere.
When $S$ is contained in a hemisphere, it is local $H^-$ rigid in a certain dumb-bell shaped domain $\Theta$.
\end{theorem}

In this follow-up work, we will mainly discuss those rigidity properties on doubly connected domains symmetric about the equator on the unit sphere $S^2\subset \mathbb{R}^3$.

\textbf{Convention:}\\
(1) $S^2 \subset \R^3$ is the unit sphere with radius 1, defined by $$S^2=\{ (x_1,x_2,x_3) \in \R^3 | x_1^2+x_2^2+x_3^2 =1\}$$
$S^2$ divides $\R^3$ into two connected components, and $D^3 = \{ (x_1,x_2,x_3) \in \R^3 | x_1^2+x_2^2+x_3^2 \le 1\}$ is one of them. Also, we define $P_i$ the coordinate hyperplane in $\R^3$, with $$P_i = \{ (x_1,x_2,x_3) \in \R^3 | x_i = 0\}$$
In this passage we will usually consider some curves in $P_2$, and we will use $(x_3,x_1)$ as the coordinate in $P_2$.\\
(2) Suppose $a \in (0,1)$ and $S_a$ is an annulus around the equator with width $2\arccos a$, i.e. $$S_a = \{ (x_1,x_2,x_3) \in S^2 | x_1^2+x_2^2 > a^2 \}$$
And we use $\Sigma_a$ to refer to a perturbation of $S_a$.\\
(3) Define $\widetilde{\pmb{n}}(x_1,x_2,x_3) = (-x_1,-x_2,-x_3)$ as the unit inward normal vector field on $S^2$. In this passage, if there is no other explanation, we will default that the unit normal vector fields of $\Sigma_a$ we talk about are all consistent with $\widetilde{\pmb{n}}$ at $S^2 \backslash S_a$, which is inward. And the mean curvature of $S_a$ and $\Sigma_a$ also come from it.

The main results in this paper are stated as below.

\begin{theorem}\label{thm-rigi}
For $a \in (0,1)$, we have:\\
(1) $S_a$ is $H^+$ rigid iff $a\ge \sqrt{3} / 2$.\\
(2) $\forall a \in (0,1)$, $S_a$ is not $H^-$ rigid.
\end{theorem}

\begin{theorem}\label{thm-rigi-local}
There exists a constant $a_0 \approx 0.5524$ such that for $a \in (0,1)$:\\
(1) $S_a$ has local $H^+$ rigidity iff $a \ge a_0$.\\
(2) $S_a$ has local $H^-$ rigidity iff $a> a_0$.
\end{theorem}

This paper is organized as follows. In Section~2, we review the trap-slice lemma in \cite{ChenMaWang} and the Tangency Principle (see also \cite{Fontenele} and \cite{Souam}). Together with suitably chosen trap and comparison surface we establish the strong $H^+$ rigidity in \ref{thm-rigi}. Then in Section~3 we establish local $H^+$($H^-$) rigidity by detailed analysis of the related ODE. The round-corner lemma is established in Section~4, which is applied to a gluing construction using Delaunay surfaces to find non-trivial deformations increasing or decreasing the mean curvature, hence establish the \emph{only if} part of the above two theorems. This finishes the proof to the main theorems. Some technical details involving elliptical integrals are left to the appendix.

\par
\noindent
\textbf{Acknowledgement.} \par
We would like to thank Yichen Cheng, who participated in our seminar and suggested the construction of trap (comparison surfaces) using two spheres in Step 3 of the proof of Theorem~\ref{Prigi}. We also thank other participants of our seminar: Zheng Yang, Yi Sha, Tianming Zhu, Shunkai Zhang, for their interests and many helpful discussions. This research project is partially supported by NSFC grant 11831005 and the Undergraduate Student Research Study program of Peking University for the year 2021-2022.

\section{The trap-slice lemma and the strong rigidity}

The Tangency Principle \cite{Fontenele, Souam} is an important instrument for mean curvature rigidity problems.

\begin{theorem}\label{thm-TP1}[The Tangency Principle]
Let $M^n_1$ and $M^n_2$ be hypersurfaces of $N^{n+1}$ that are tangent at $p$ and let $\eta_0$ be a unit normal vector of $M_1$ at $p$. Denote by $H_r^i(x)$ the $r$-mean curvature at $x\in W$ of $M_i, i=1,2,$ respectively. Suppose that with respect to this given $\eta_0$, we have:
\begin{enumerate}
 \item Locally $M_1\ge M_2$, i.e., $M_1$ remains above $M_2$ in a neighborhood of $p$;
 \item $H_r^2(x)\ge H_r^1(x)$ in a neighborhood of zero for some $r, 1 \le r \le n$; if $r\ge 2$, assume also that $M_2$ is $r$-convex at $p$.
\end{enumerate}
Then $M_1$ and $M_2$ coincide in a neighborhood of $p$.
\end{theorem}

\begin{corollary}\label{tangawa}
For $a \in (0,1)$, suppose $\Sigma_a$ is a perturbation of $S_a$, then:\\
(1)If $H(\Sigma_a) \ge 1$ and $\Sigma_a \neq id$, then $\Sigma_a \cap \mathring{D^3} \neq \varnothing$.\\
(2)If $H(\Sigma_a) \le 1$ and $\Sigma_a \neq id$, then $\Sigma_a \cap (D^3)^c \neq \varnothing$.
\end{corollary}

\begin{proof}
We only prove (1), and the proof of (2) is similar.

Consider the collection $$S_T = \{ x \in S^2 : \Sigma(x) \in S^2\}$$
It is apparent that $S_T$ is closed in $S^2$. If $\Sigma_a \cap \mathring{D^3} = \varnothing$, then for all $x \in S_T$, $\Sigma$ will be tangent with $S^2$ at $x$. Hence, from Tangency Principle, $\Sigma_a$ will coincide with $S^2$ in a neighborhood of $x$, showing that $S_T$ is also open in $S^2$. Therefore, we have $\Sigma_a = id$, which is a contradiction.
\end{proof}

The trap-slice lemma is an encapsulated version of the Tangency Principle, which was first established in our previous work \cite{ChenMaWang}.

\begin{theorem}\label{trap-slice-lemma}[The trap-slice lemma]

Let the \emph{trap} $\Omega \subset \R^n$ be a domain enclosed by two connected hypersurfaces $B_0,B_1$ sharing a boundary $A=B_0\cap B_1$ and $\partial\Omega=B_0\cup B_1$.

The \emph{slice} is a foliation of $\Omega$ by a one-parameter family of hypersurfaces $\{F_{t}\}\subset \Omega$ (with or without boundaries). When $\partial F_t\ne \emptyset$, we assume $\partial F_t \subset B_1$. Each $F_{t}$ divides $\Omega$ into two sub-domains, one having $B_0$ on its boundary, and $\Omega_t$ is the other one away from $B_0$.

Fix a real constant $\alpha\in\mathbb{R}$.
With respect to the outward normal of $\partial\Omega_t\supset F_t$, suppose that the mean curvature function of $F_{t}$ always satisfies $H(F_{t})\ge \alpha$.

Given the trap and the slice as above, there does NOT exist any hypersurface $\Sigma_*$ with boundary $\partial\Sigma_*$ satisfying all of the following conditions:
\begin{enumerate}
  \item $\Sigma_*$, the interior of the compact hypersurface $\overline\Sigma_*=\Sigma_*\cup \partial\Sigma_*$, is embedded in $\Omega$ with boundary $\partial\Sigma_*\subset B_0\subset \partial\Omega$. In particular, $\Sigma_*$ divides $\Omega$ into two sub-domains; sub-domain $\Omega_*$ is the one of them that having $B_1$ on its boundary. We orient $\Sigma_*$ by the outward normal of $\partial\Omega_*$.
  \item The boundary $\partial\Sigma_*$ has a neighborhood $U_t$ in $\overline\Sigma_*$ not contained in $\Omega_t$ for any $t$.
  \item Given the orientation of $\Sigma_*$, the mean curvature function $H(\Sigma_*)\le \alpha$.\\
\end{enumerate}
\end{theorem}

\begin{corollary}\label{cor-trap-slice}\cite{ChenMaWang}
Assumptions on the trap $\Omega\subset \R^n, \partial\Omega=B_0\cup B_1$ and the slice $\{F_t\}$
are as in the trap-slice lemma (Theorem~\ref{trap-slice-lemma}).
Moreover, we suppose that:
\begin{enumerate}
  \item $B_0$ is also one leave of the foliation $\{F_t\}$ (we may suppose $B_0=F_0$ is an open subset of $\partial\Omega$);
  \item For any other $t\ne 0$, either $\partial B_0 \cap \partial F_t=\emptyset$, or $B_0$ intersects with $F_t$ at their boundaries transversally.
\end{enumerate}
Then $B_0$ admits no non-trivial perturbation $\Sigma_0$ (with fixed boundary up to $C^2$ and the same orientation on $\partial\Sigma=\partial B_0$) such that $H(\Sigma_0)\le \alpha$, unless two hypersurfaces $\Sigma_0$ and $B_1$ intersect at their interior points.\\
\end{corollary}

\begin{remark}\label{remark-rconvex}
The trap-slice lemma and Corollary~\ref{cor-trap-slice} above are still true when the assumptions are changed as below:
$\Sigma_*$ and $F_t$ are oriented by the inward normal vectors with respect to $\Omega_*$ and $\Omega_t$, respectively,
and the inequality on $H$ is reversed as
\[H(F_t)\le \alpha\le H(\Sigma_*).\]
\end{remark}

Now we consider the $H^+$ rigidity of $S_a$:

\begin{theorem}\label{Prigi}
Suppose $a \in [\sqrt{3} / 2 , 1)$, then $S_a$ has $H^+$ rigidity.
\end{theorem}

\begin{proof}
From Remark~\ref{monotonicity}, we only need to consider $a = \sqrt{3} / 2$. Assuming there is a perturbation $\Sigma_a \neq id$ of $S_a$ such that $H(\Sigma_a) \ge 1$, we will try to find contradiction.\\

\noindent
Step 1: Denote $B_1 = S^2 \cap \{ x_3 \le -1/2 \}$, and $B_0$ the symmetrical surface of $B_1$ with respect to $x_3 = -1/2$. They enclose an open domain $\Omega \subset D^3$, which is our "trap".\\
Then we translate $B_0$ by the vector $\pmb{v}_t = (0,0,-t),0 \le t < 1$, denoted the translated surface as $B_t$. Denote $F_t = B_t \cap \Omega$, which is our \emph{slice}. The normal of $B_0$ and $F_t$ are all inward about $\Omega$.

We assert that $\Sigma_a \cap \Omega = \varnothing$, because if not, we can choose a connected component of $\Sigma_a \cap \Omega$ and denote it $\Sigma_*$. Might as well, assume there exist $p_1 \in \Sigma_*$ and $p_2 \in B_1$ such that the open line segment $p_1p_2 \cap \Sigma_a = \varnothing$(this is reasonable for $\Sigma_a$ is an embedded map of $S^2$). Then the normal on $\Sigma_*$ will suit the condition 1 in trap-slice lemma and Remark~\ref{remark-rconvex}. 

Also, it is apparent that the boundary $\partial \Sigma_* \subset B_0$ suits condition 2 in trap-slice lemma, since $\partial F_t \subset B_1$ and $\Sigma_a$ is an embedded map. Hence, we get the contradiction by Remark~\ref{remark-rconvex}.

 \begin{figure}[!h]
   \setlength{\belowcaptionskip}{0.1cm}
   \centering
   \includegraphics[width=0.56\textwidth]{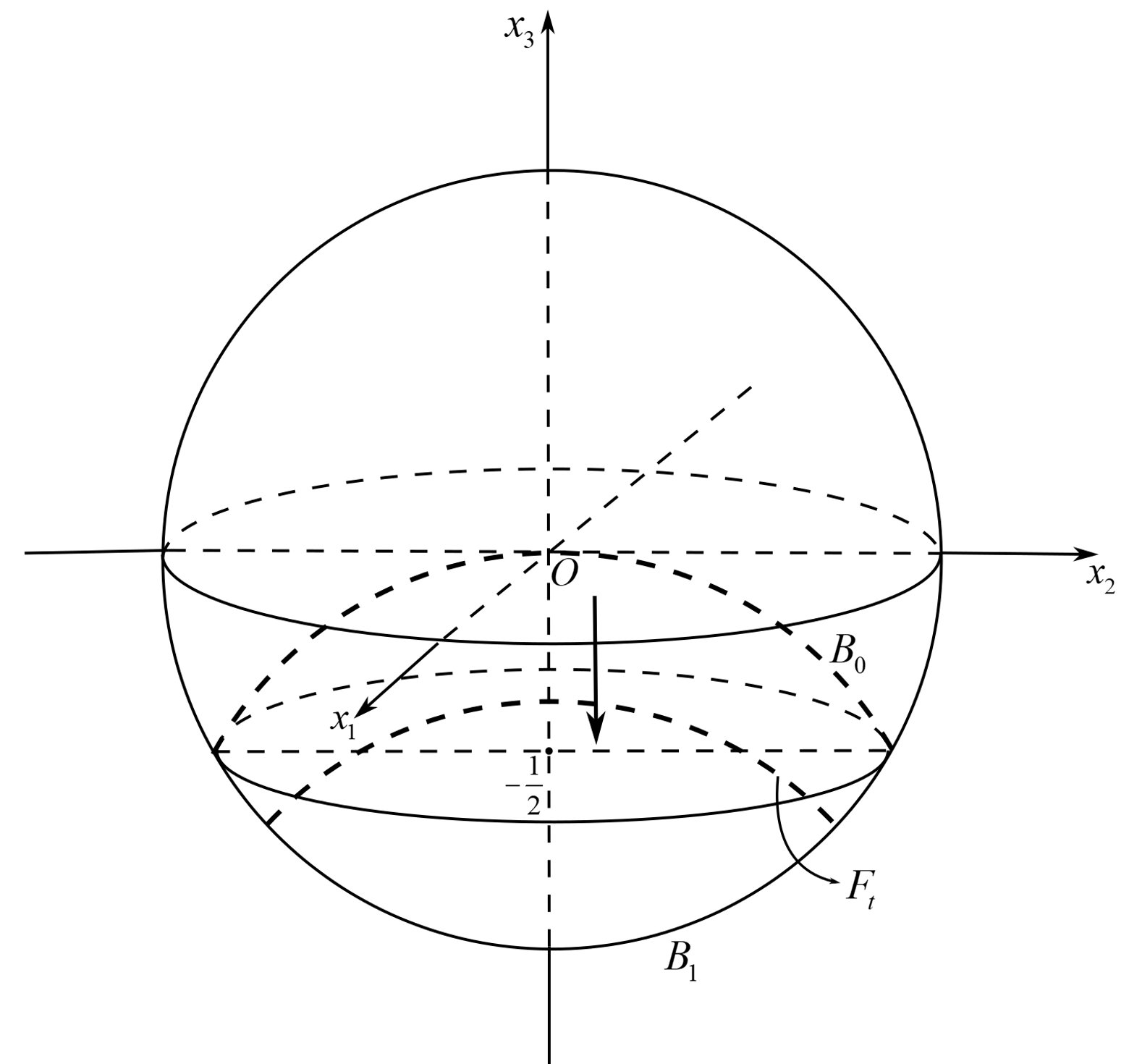}
   \label{fig1}
 \end{figure}

Similarly, denote $\tilde{B}_1 = S^2 \cap \{ x_3 \ge 1/2 \}$, and $\tilde{B}_0$ the symmetrical surface of $\tilde{B}_1$ with respect to $x_3 = 1/2$. They enclose an open domain $\tilde{\Omega} \subset D^3$, which is symmetrical with $\Omega$ with respect to $P_3$. We can also get $\Sigma_a \cap \tilde{\Omega} = \varnothing$.\\

\noindent
Step 2: Since we have had
\begin{align}\label{assert1}
\Sigma_a \cap (\Omega \cup \tilde{\Omega}) = \varnothing,
\end{align}
we will then further consider where $\Sigma_a$ is.

From Corollary~\ref{tangawa}, we know $\Sigma_a \cap \mathring{D^3} \neq \varnothing$. Hence, we can select a connected component of $\Sigma_a \cap D^3$, denoted as $\Sigma^*$.

We can prove that 
$$\Sigma^* \cap \{ x_1 ^2 + x_2 ^2 < \frac{1}{4} \} \neq \varnothing.$$
In $P_2$, define $\gamma$ as 
$$\gamma = \{ (x_3,x_1) \in P_2 : x_1 ^2 + x_3 ^ 2 - 2|x_3| = 0, |x_3| \le \frac{1}{2}, x_1 \ge 0 \}.$$
For $t \in [1/2 , \sqrt{3} / 2]$, define the line segment 
$$l_t = \{ x_1 = t : |x_3| \le 1 - \sqrt{1-t^2} \}.$$
And for $t \in (\sqrt{3} / 2 , 1)$, define $l_t$ the minor arc segment connecting $(-1 / 2,\sqrt{3}/ 2 )$, $(0,t)$ and $(1 / 2, \sqrt{3} / 2)$ in $P_2$. Then we rotate $l_t$ around $x_3$ axis to generate the slice $F'_t$.

 \begin{figure}[!h]
   \setlength{\belowcaptionskip}{0.1cm}
   \centering
   \includegraphics[width=0.56\textwidth]{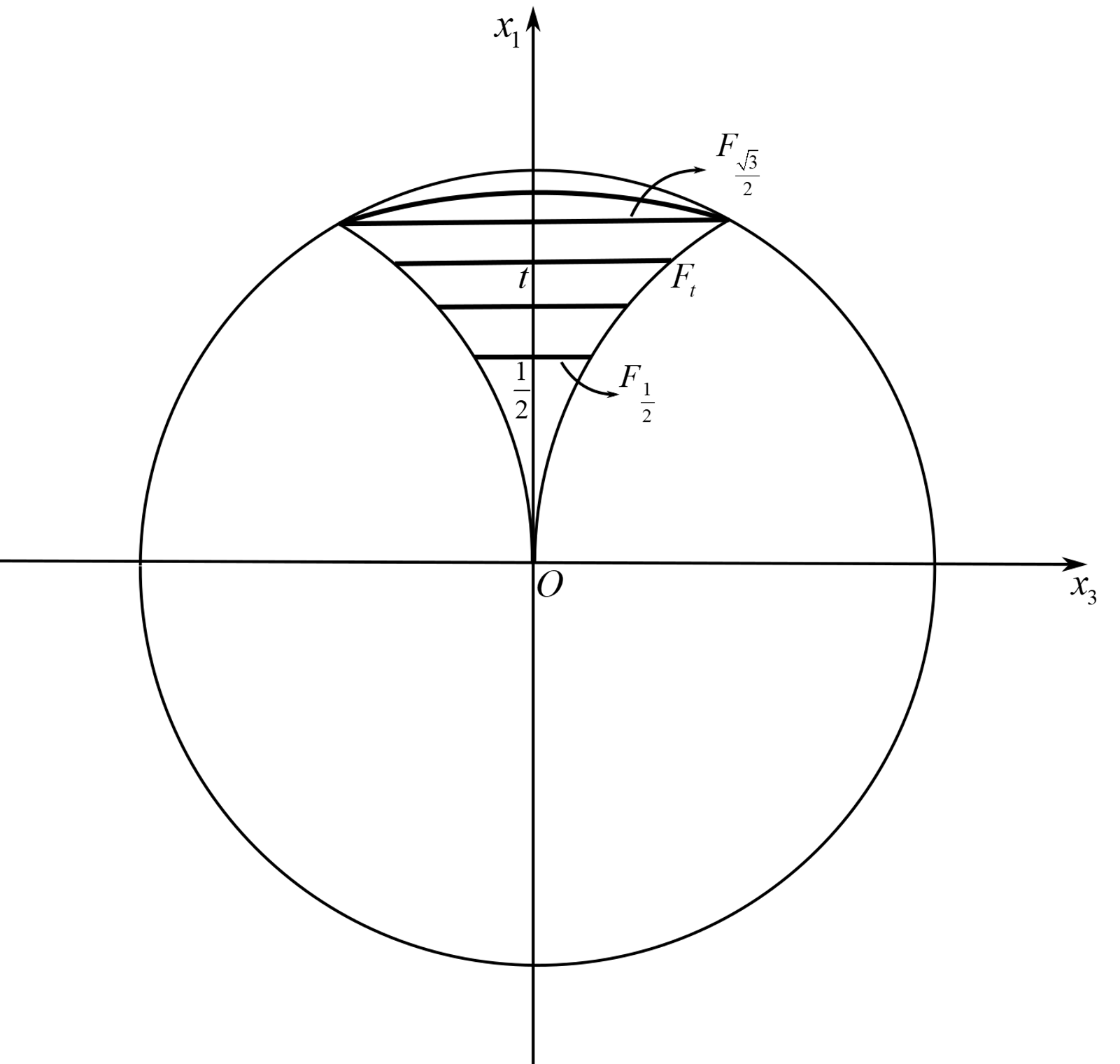}
   \label{fig2}
 \end{figure}

It can be easily verified that
 $$H(F'_t) = \frac{1}{2t} < 1, \frac{1}{2} \le t \le \frac{\sqrt{3}}{2}. $$
And when $\sqrt{3}/ 2 < t < 1$, consider the radius of $l_t$, denoted as $r_t$, with attention that $r_t > 1$. For $\forall x_3 \in [-1 / 2, 1/2]$, define $H^t(x_3)$ as the mean curvature of $F'_t$ at $(x_3, l_t(x_3))$, and it can be easily verified that 
$$H^t(x_3) \le H^t (\frac{1}{2}) = \frac{1}{2} (\frac{1}{r_t} + \frac{\sqrt{4r_t^2 - 1}}{\sqrt{3}r_t}) < 1.$$
Hence, for $F'_t$ as our second "slice", there is 
$$H(F_t') \le 1, \forall t \in (\frac{1}{2} , 1).$$
Define $\omega' \subset P_2$ the open domain surrounded by $\gamma$, $l_{\frac{1}{2}}$ and $l_1$. Rotate $\omega'$ around $x_3$ axis, generating a domain $\Omega' \subset \R^3$, which is our second "trap". Also, define $B_0' = F'_1$ and $B_1' = \partial \Omega' \backslash B_0'$.

If $\Sigma^* \subset \{ x_1 ^2 + x_2 ^2 \ge 1 / 4 \}$, we can consider $\Omega'$, $B_0'$, $B_1'$ and $F'_t(1/2 \le t \le 1)$, and from Remark~\ref{remark-rconvex}, we get the contradiction.\\

\noindent
Step 3: In $P_2$, we define a series of undulary as $u_t(0 < t< 1/2)$ in $|x_3| \le 1 /2$, whose neck is $(0,t)$. Consider a series of elliptic in $P_2$, defined as
\begin{align}\label{ellip}
E_t: \frac{x_3^2}{t - t^2} + (2x_1 - 1)^2 = 1.
\end{align}
It is apparent that the focal points of $E_t$ are $(0,t)$ and $(0,1-t)$, and define $u_t$ as the orbit of $(0,t)$ when $E_t$ rotates towards right along the $x_3$ axis.

Also, it is known that if we rotate $u_t$ around $x_3$ axis to generate a series of CMC surface, denoted as $U_t$, then $U_t$ are all CMC surface segments, satisfying $$H(U_t) =1,\forall t \in (0, \frac{1}{2}).$$
Then, we will prove the lemma below:

\begin{lemma}\label{Tundu}
for $\forall t \in (0, 1/2)$, define $x_U(t) > 0$ such that $x_U(t) =u_t(1 / 2)$. Then we always have $x_U(t) < \frac{\sqrt{3}}{2}$
\end{lemma}

\begin{proof}
Still consider the rotation of $E_t$ defined by~\ref{ellip}. As the figure below, when slope of the long axis of $E_t$ is $- \sqrt{3} / 2$, as the figure shows, denote the ellipse as $E'_t$, and then define $P$ as the tangent point of $E'_t$ with $x_3$ axis, $A_1 \in u_t$ as the focal point of $E'_t$ rotated from $(0,t)$, $A_2$ as the other focal point, $BC$ as the long axis of $E'_t$, $D \triangleq BC \cap x_3$ axis, and $E \in x_3$ axis such that $AE \perp x_3$ axis. \\
Firstly, consider $A_2'$ the symmetric point of $A_2$ about $x_3$ axid. So $A_1, P, A_2'$ are collinear, with $$|A_1 A_2'| = |A_1P| + |A_2P| = 1$$
Hence, we know $|A_1 D| < |A_1A_2'| = 1$ from $\angle A_1 D A_2' = 2\pi / 3 > \pi / 2$. So we have $$|A_1 E| < \frac{\sqrt{3}}{2} \quad |ED| < \frac{1}{2}$$
Then, since $|OE|$ is the length of the minor elliptic arc from $B$ to $P$, we have $|OE| > |BP|$, hence $$|OD| > |BP|+|PD| > |BD| > 1$$
Therefore, we know $|OE| = |OD| - |ED| > 1 / 2$, which tells us that $$u_t(\frac{1}{2}) < u_t(x_3(A_1))  = |A_1 E| < \frac{\sqrt{3}}{2}$$
 \begin{figure}[!h]
   \setlength{\belowcaptionskip}{0.1cm}
   \centering
   \includegraphics[width=0.62\textwidth]{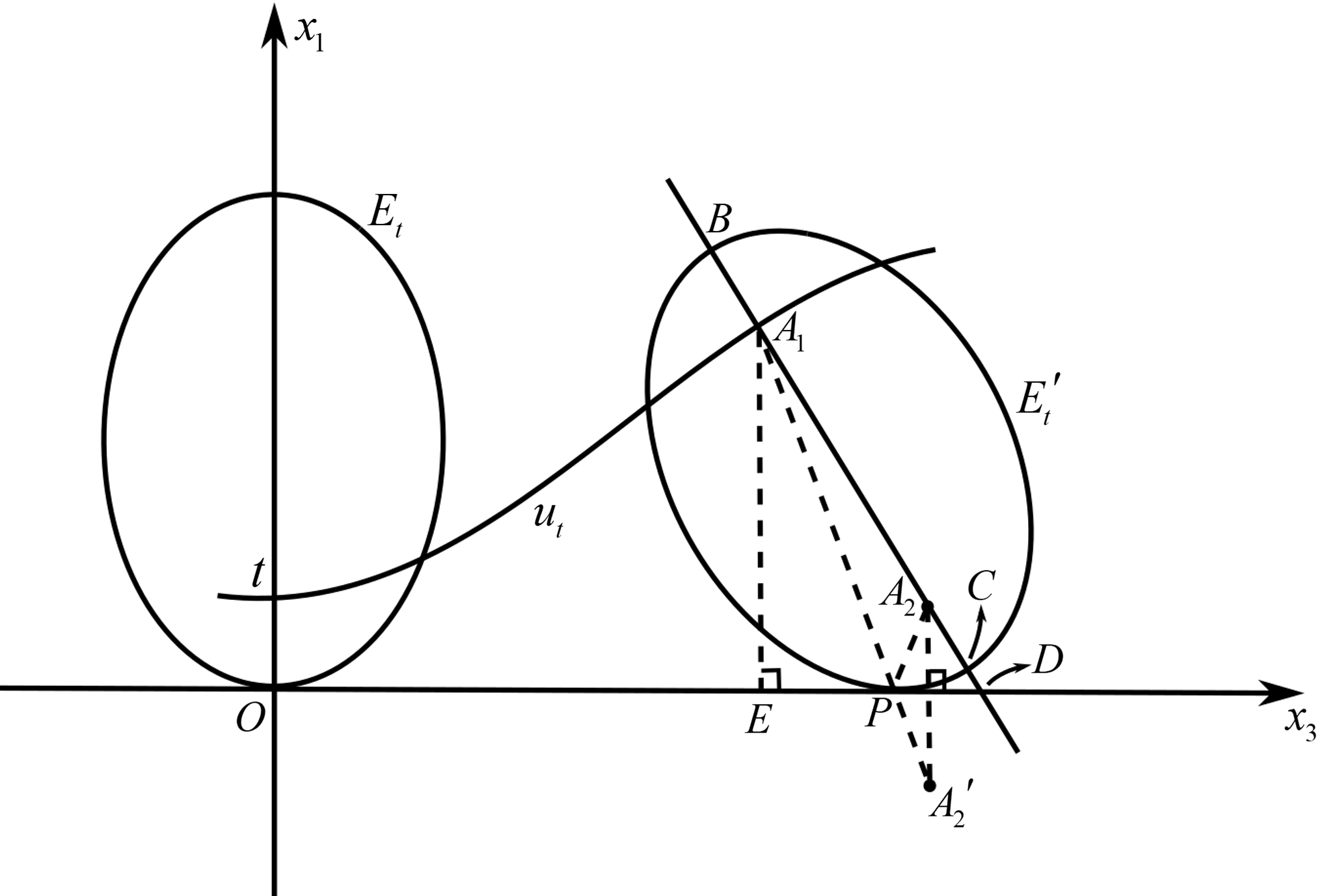}
   \label{fig3}
 \end{figure}

Thus we finish the proof of this lemma.
\end{proof}

Now select $\Sigma^{**}$ as one connected component of $\Sigma^* \cap \{ x_1 ^2 + x_2 ^2 \le 3 / 4 \}$ such that $$\Sigma^{**} \cap \{ x_1 ^2 + x_2 ^2 < \frac{1}{4} \} \neq \varnothing$$
We can assert that
\begin{align}\label{assert2}
\Sigma^{**} \cap \partial S_a = \varnothing.
\end{align}
Actually, this assertion will directly come from the basic proposition below, since $\Sigma^{**}$ is connected.

\begin{proposition}\label{partial}
For $\forall p \in \partial S_a \subset S^2$, we can find a domain $B \subset S^2$ with $p \in B$, such that $$\Sigma(B \backslash \overline{S_a}) \subset \{ x_1^2 + x_2^2 > \frac{3}{4} \}$$
\end{proposition}

\begin{proof}
We will define a vector field $X$ near $\partial S_a$. If we use the spherical coordinates in $S^2 \backslash (0,0, \pm 1)$: 
$$(\theta, \phi) \mapsto (\sin \theta \cos \phi , \sin \theta \sin \phi , \cos \theta).$$
We define a tangent vector field on $S^2 \backslash (0,0,\pm 1)$ as 
$$X(\theta, \phi) = (\cos \theta \cos \phi , \cos \theta \sin \phi , -\sin \theta).$$
Define $\Sigma_T$ as the tangent map of $\Sigma$, then $\Sigma_T X$ can be seen as a smooth vector field on $\Sigma$ near $\partial S_a$.
 
We can select the geodesic circles $G_{\phi}$ passing $(0,0,\pm 1)$ in $S^2$, whose $\phi$ coordinate is constant. They can be seen as the integral curves of $X$, so $\Sigma(G_{\phi})$ is the integral curves of $\Sigma_T X$.
 
Define $\pmb{v}_1 = (0,0,-1)$, and might as well, assume that $x_3(p) > 0$. Since $\Sigma_T X$ is smooth and $\Sigma_T = id$ in $S^2 \backslash S_a$, select a domain $p \in B \subset S^2$ such that for $\forall q \in B$, we have $$\Sigma_T X_q \cdot \pmb{v}_1 \approx \frac{\sqrt{3}}{2}> 0,$$
and
\begin{align}\label{Txq}
\Sigma_T X_q - (\Sigma_T X_q \cdot \pmb{v}_1)\pmb{v}_1  \approx (\frac{\cos \phi}{2}, \frac{\sin \phi}{2}, 0)
\end{align}
which is outward.

Hence, if $q \notin \overline{S}_a$, take the intergral curve of $T_{\Sigma} X$ passing $\Sigma(q)$, denoted as $I_q$, with $I_q(0) \in \partial S_a \cap B$, and it is apparent that $x_1^2(I_q) + x_2^2(I_q)$ increases from~\ref{Txq}. Thus we have $x_1^2(\Sigma(q)) + x_2^2(\Sigma(q)) > 3 /4$, which is what we want.
\end{proof}

Now define $B'' = B_0 \cup \tilde{B}_0$, and from $\Sigma^{**} \cap \{ x_1 ^2 + x_2 ^2 < 1 / 4 \} \neq \varnothing$, we know $\exists t_0^{**} \in (0, 1/2)$, such that $\Sigma^{**} \cap U_{t^{**}_0} \neq \varnothing$. Define $t^{**} = \inf \{ t \in (0, 1/2) | \Sigma^{**} \cap U_t \neq \varnothing \}$.

If $t^{**} > 0$, then from lemma~\ref{Tundu}, we know $U_{t^{**}} \cap \partial \Sigma^{**} = \varnothing$, so $U_{t^{**}}$ must be tangent with $\Sigma^{**}$ at their intersect points. Hence, from the Tangency Principle, we know $\Sigma^{**} \cap U_{t^{**}}$ must be both open and closed in $\Sigma^{**}$, which is impossible.
 \begin{figure}[!h]
   \setlength{\belowcaptionskip}{0.1cm}
   \centering
   \includegraphics[width=0.56\textwidth]{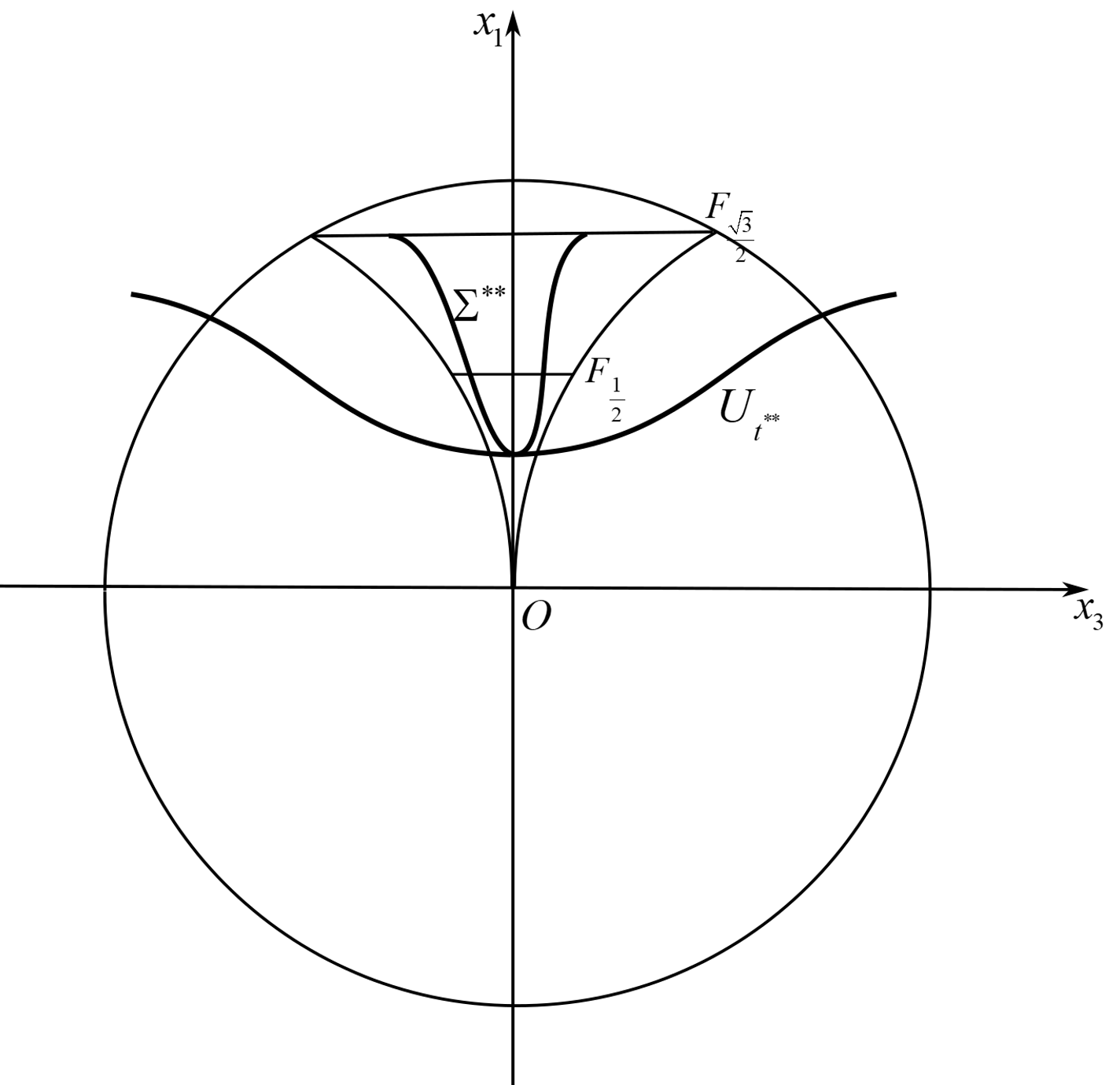}
	\caption*{$\Sigma^{**}$ and selection of $U_{t^{**}}$ as the profile in $P_2$}
   \label{fig4}
 \end{figure}

If $t^{**} = 0$, since the assertion~\ref{assert2}, we know that $\Sigma^{**}$ must be tangent with $B''$ at some intersect points. Also, it is apparent that $O \notin \Sigma^{**}$ from the regularity of $\Sigma$ and ~\ref{assert1}. Hence, similar to the condition of $t^{**} >0$, we also get the contradiction.

As a result, we finish the proof of Theorem~\ref{Prigi}.
\end{proof}

\section{The local rigidity results}

The trap-slice lemma is still the main tool for the local rigidity problem. What we need is to construct the surface for comparison, i.e. the slices, in a suitable trap (which is almost the region $\Theta$ in the definition of local regidity.)

To construct such slices near $S_a$, we will consider a local family of CMC surface pieces $\{ \widetilde{C}(a,t) \}$ near $S_a$ with the same boundary. For this, we will first discuss the features of those generatrices of Delaunay surfaces in $P_2$, which will generate a series of CMC surface in $\R^3$, i.e. Delaunay surfaces. In order to discribe their features, we turn to the ODE determining them:

\begin{proposition}\label{delaunay}
Fix $a \in (0,1)$. In $P_2$, consider the system when $x_3  \ge 0$:

\begin{align*}
\frac{dx_1}{dx_3} = -\sqrt{(\frac{x_1}{H x_1^2 +t-H t^2})^2 - 1}\\
x_1(0) = t
\end{align*}
where $H$, $t$ are the parameters satisfying $H \approx 1$ and $t \approx 1$. It has the unique solution that strictly decreases, which can be written near $(0,1) \in P_2$ with $0 < x_1 < t, x_3 > 0$, as
\begin{align}\label{basic-ode}
x_3 =\int_{x_1}^{t} [(\frac{x_1}{H x_1^2 +t-H t^2})^2 - 1]^{-\frac{1}{2}} dx_1
\end{align}

\begin{enumerate}
  \item We can find $\delta_a > 0$ and $\epsilon_a \ll 1$, such that if $|H-1|+ |t-1| < \epsilon_a$, the solution~\ref{basic-ode} is well defined in $x_3 \in [0, \sqrt{1-a^2} + \delta_a]$, denoted as $x_1 = c(H,t,x_3)$, and we use $c(H,t)$ as the abbreviation of this function.
  \item Denote the even extension of $c(H,t)$ about $x_1$ axis still as $c(H,t)$. In $\R^3$, rotate each $c(H,t)$ around $x_3$ axis, and it will create a Delaunay surface piece $C(H,t)$, which is CMC. The mean curvature of $C(H,t)$(with the inward normal) is exactly $H$.
\end{enumerate}
\end{proposition}

\begin{proof}
The first of this proposition will be apparently guaranteed since $x_1 = c(H,t,x_3)$ is continuous of $H$ and $t$ and $c(1,1)$ is exactly the semicircle. The second comes from the basic formula: $$H(x_3) = (\frac{c}{\sqrt{1+c'^2}})' / (c^2)'$$
Where $c$ is any function $x_1 = c(x_3)$, and the derivation is to $x_3$, and $H(x_3)$ is the mean curvature of the surface generated by rotating $c$ around $x_3$ axis, with the inward normal. Then the second can be easily verified from this formula and the defination of $c(H,t)$.
\end{proof}

\begin{corollary}\label{ode}
We define:
\begin{align}\label{D}
D(H,t,x) = \sqrt{(\frac{x}{H x^2 +t-H t^2})^2 - 1}
\end{align}
Then the ODE of $c(H,t)$ can be shown as
\begin{equation*}
\int_{x_1}^{t} \frac{1}{D(H, t, x_1)}dx_1 = x_3
\end{equation*}
\end{corollary}

\begin{remark}
For $H \approx 1, t \approx 1$ in proposition~\ref{delaunay}, $c(H,t)$ coincides for all different $a$ when all the parameters are in the domain of definition, which tells us that we do not need to set $a$ as one parameter for $c$.
\end{remark}

It is apparent that we can find $\epsilon'_a \in (0, \epsilon_a)$ such that if $|H-1|+|t-1| < \epsilon'_a$, there exists unique function $x^*(a,H,t) \in (\sqrt{1-a^2} - \delta_a, \sqrt{1-a^2} + \delta_a)$, defined as $$x^*(a,H,t) = c^{-1}(H,t)(a)$$

Now we will consider an important feature of $x^*$, which will help us consider some monotonicity of $c(H,t)$:

\begin{proposition}\label{x3mono}
For $x^*(a,H,t)$ with well defined parameters, We have $$\frac{\partial x^*}{\partial H} < 0.$$
\end{proposition}

\begin{proof}
Directly from~\ref{D}, it is easy to verify that $$\frac{\partial}{\partial H} D(H,t,x)> 0 \quad \forall x \in (x_1 ,t).$$
Hence, this proposition it trivial from
\begin{align}\label{x*}
x^*(a,H,t) = \int_{a}^{t} \frac{1}{D(H, t, x_1)}dx_1.
\end{align}
\end{proof}

\begin{remark}
Fix $a \in (0,1)$, and consider the equation of $H$ and $t$ below:
\begin{align}\label{Ha}
x^*(\sqrt{1-a^2}, H,t) = a.
\end{align}
Then from the implicit function theorem, $\exists \epsilon''_a \in (0, \epsilon'_a)$, such that~\ref{Ha} can be seen as a function $H = H(t), |t-1| \le \epsilon''_a$. We denote this function as $H^a$.
\end{remark}

After those preparations, we can introduce the generatrix of the surface of comparison we need. Define
\begin{align}\label{tildec}
\tilde{c}(a,t) = c(H^a(t), t)
\end{align}
and $$\widetilde{C}(a,t) = C(H^a(t), t) \cap \{ |x_3| \le \sqrt{1-a^2} \}.$$
It is apparent that $\tilde{c}(a,t, \sqrt{1-a^2}) = a$.

We will use $\tilde{c}$ to generate the "slice" we need, but we do not know whether $H^a(t)$ will increases or decreases near $t = 1$, which is the key for the local $H$ rigidity. Actually, we can see from below that $a$ will influence the monotonicity of $H^a$ near $t= 1$.\\
Define a constant $a_0 \in (1/2, 1)$ as the unique null point of the function $$g(a) \triangleq -\ln \frac{1+\sqrt{1-a^2}}{a} +  \frac{1}{\sqrt{1 - a^2}}$$
It can be easily extimated that $a_0 \approx 0.5524$. Then we have the lemma below:

\begin{lemma}\label{Htwith1}
For $a \in (0,1)$, there exists $\tilde{\epsilon}_a \in (0, \epsilon''_a)$, such that:\\
(1) If $0 < a < a_0$, then $H^a(t) > 1, 1- \tilde{\epsilon}_a \le t <1$, and $H^a(t) < 1, 1 < t \le 1+ \tilde{\epsilon}_a$\\
(2) If $a = a_0$, then $H^a(t) < 1, 0 < |t-1| \le \tilde{\epsilon}_a$\\
(3) If $a_0 < a < 1$, then $H^a(t) < 1, 1- \tilde{\epsilon}_a \le t <1$, and $H^a(t) > 1, 1 < t \le 1+ \tilde{\epsilon}_a$.
\end{lemma}

\begin{proof}
It is difficult to directly consider $dH^a / dt$, but we will introduce another lemma about $c(1,t)$ to assist us.

\begin{lemma}\label{H=1}
Define $x_a(t) = x^*(a,1,t)$. Then there exists $0 < \eta_a \ll 1$ such that:\\
(1)When $a > a_0$, then $x_a < \sqrt{1 - a^2}, 1-\eta_a \le t < 1$, and $x_a > \sqrt{1 - a^2}, 1 < t\le 1+\eta_a$\\
(2)When $a < a_0$, then $x_a > \sqrt{1 - a^2}, 1-\eta_a \le t < 1$, and $x_a < \sqrt{1 - a^2},1 < t \le 1+\eta_a$\\
(3)When $a = a_0$, then $x_a < \sqrt{1 - a^2}, 0 <|t-1|\le \eta_a$
\end{lemma}

The proof of lemma~\ref{H=1} will be put in the Appendix.

It can be noticed that the inequality sign in these two lemmas are consistent. Actually, we can prove this consistency, which will finish the proof of lemma~\ref{Htwith1}.

It is apparent from ~\ref{x*}, ~\ref{Ha} and ~\ref{tildec} that
\begin{align*}
\sqrt{1-a^2} = x^*(a,H^a(t),t) \\
x_a(t) = x^*(a,1 ,t)
\end{align*}
And from lemma~\ref{x3mono}, it is apparent that $$x_a < \sqrt{1-a^2} \Rightarrow H^a(t) < 1$$
$$x_a > \sqrt{1-a^2} \Rightarrow H^a(t) > 1$$
Hence, if we choose $\tilde{\epsilon}_a = \min \{ \eta_a, \epsilon''_a \}$, then lemma~\ref{Htwith1} will be directly proved from lemma~\ref{H=1}, and this is what we need.
\end{proof}

\begin{corollary}\label{yoko}
If $a \in [a_0, 1)$ and $t \in [1-\tilde{\epsilon}_a , 1)$, then at the same $x_1$ coordinate, we have $$\frac{d\tilde{c}(a, t)}{dx_3} |_{x_1} > \frac{d\tilde{c}(a, 1)}{dx_3} |_{x_1} \quad \forall x_1 \in [a, t]$$
And if $a > a_0$ and $t \in (1, 1+\tilde{\epsilon}_a]$, then we have $$\frac{d\tilde{c}(a, t)}{dx_3} |_{x_1} < \frac{d\tilde{c}(a, 1)}{dx_3} |_{x_1} \quad \forall x_1 \in [a, 1].$$
\end{corollary}

\begin{proof}
By Proposition~\ref{delaunay}, for $a \in [a_0, 1)$ and $t \in [1-\tilde{\epsilon}_a , 1)$, what we only need to verify is $$D(H^a(t), t, x_1) < D(1,1,x_1) \Leftrightarrow \frac{x_1}{x_1^2H^a(t)+t-t^2H^a(t)} < \frac{1}{x_1}$$
Since $H^a(t) < 1$ and $a\le x_1 \le t < 1$, we have $$t - t^2 H^a(t) > t^2(1-H^a(t)) \ge x_1^2(1-H^a(t))$$
which shows the first relation of this corollary is true.\\
Similarly, if $a > a_0$, for $t \in (1, 1+\tilde{\epsilon}_a]$, we only need $$D(H^a(t), t, x_1) > D(1,1,x_1) \Leftrightarrow \frac{x_1}{x_1^2H^a(t)+t-t^2H^a(t)} > \frac{1}{x_1}$$
But this time we have $H^a(t) > 1$ and $a\le x_1 \le 1 < t$, hence $$t - t^2 H^a(t) < t^2(1-H^a(t)) \le x_1^2(1-H^a(t))$$
Hence, we prove the second relation similarly.
\end{proof}

\begin{remark}\label{htmono}
From Corollary~\ref{yoko}, and using basic knowledge of ODE, we can easily get $$\tilde{c}(a,t) < \tilde{c}(a,1), \forall a \in [a_0, 1), t \in [1-\tilde{\epsilon}_a , 1)$$ and $$\tilde{c}(a,t) > \tilde{c}(a,1), \forall a \in (a_0, 1), t \in (1, 1+\tilde{\epsilon}_a].$$
\end{remark}

After that, we can consider local $H$ rigidity:

\begin{theorem}\label{localH-}
For $\forall a \in (a_0,1)$, $S_a$ has the local $H^-$ rigidity.
\end{theorem}

\begin{proof}
Define $a' = (a_0 + a) / 2$, and it comes from Remark~\ref{htmono} that $$\tilde{c}(a', t_1) > \tilde{c}(a', 1), \forall t_1 \in (1, 1 + \tilde{\epsilon}_{a'}).$$
We fix such one $t_1$, and , and then define $$B_0 = \widetilde{C}(a',1) \quad B_1 = \widetilde{C}(a', t_1).$$
It should be noticed that both of them are defined about $a'$ rather than $a$.

Then, define $\Theta_1$ the domain enclosed by $B_0$ and $B_1$ directly, and select an open domain $\Theta \subset \R^3$ with $\Theta \setminus D^3 = \Theta_1$. We will then prove that $\Theta$ suits our requests.

Consider any perturbation $\Sigma_{a}$ with $\Sigma_{a} (S_{a}) \subset \Theta$, and $H(\Sigma_{a}) \le 1$. If $\Sigma_{a} \neq id$, from Proposition~\ref{tangawa}, we can select a connected component of $\Sigma_{a} \backslash D^3$, defined as $\Sigma_*$, and it is apparent that $\partial \Sigma_* \subset \overline{S}_a$.

It is apparent from Remark~\ref{htmono} and the continuity of $\tilde{c}(a',t)$, that $\exists \tilde{t} \in (1, t_1)$, such that $$\widetilde{C}(a', \tilde{t}) \cap \Sigma_* \neq \varnothing$$
Since $\Sigma_* \cap B_1 = \varnothing$, and from the compactness of $\overline{\Sigma}_*$ and continuity of $\tilde{c}(a',t)$, define $$t_2 = \sup \{ t \in (1, t_1) : \widetilde{C}(a',t) \cap \Sigma_* \neq \varnothing \}.$$
Hence, $\Sigma_*$ must be tangent with $\widetilde{C}(a', t_2)$ at some inner points, since $\partial \Sigma_* \subset \overline{S}_a \subset S_{a'}$.

From Lemma~\ref{Htwith1}, we know $H^{a'}(t_2) > 1$, then from Tangency Principle, $\Sigma_* \cap \widetilde{C}(a', t_2)$ must be both open and closed in $\Sigma_*$, which is a contradiction.
\end{proof}

And for local $H^+$ rigidity, we have:

\begin{theorem}
For $\forall a \in [a_0,1)$, $S_a$ has local $H^+$ rigidity.
\end{theorem}

\begin{proof}
We only need to prove $S_{a_0}$ has the local rigidity since Remark~\ref{monotonicity}, and we will also use Tangency Principle and trap-slice lemma and to finish it. However, this time we can not select some $a'$ as the proof of Theorem~\ref{localH-}, so we will need a more complex discussion. Also, we will use the same symbols as Theorem~\ref{localH-} to show a contrast.\\

\noindent
Step 1: It comes from remark~\ref{htmono} that $\tilde{c}(a_0, t_1) < \tilde{c}(a_0, 1), \forall t_1 \in (1 - \tilde{\epsilon}_{a_0}, 1)$. Fix such one $t_1$.

Now we will construct $\Theta \subset \R^3$, such that $\overline{S}_{a_0} \subset \Theta$ and $S_{a_0}$ has $H^+$ rigidity in it.

First, we can denote $B_0 = \widetilde{C}(a_0,1)$ and $B_1 = \widetilde{C}(a_0,t_1)$, and they enclose an open domain $\Theta_1 \in \R^3$.\\ Second, select $r \in (0, a_0 - 1/2)$, and define
\begin{align*}
\Theta_2 &= \cup_{\pmb{x} \in \partial S_{a_0}} B(\pmb{x}, r)\\
\Theta_0 &= (\Theta_1 \cup \Theta_2) \cap \mathring{D^3}
\end{align*}
Now select a domain $\Theta \subset \R^3$ with $\overline{S}_{a_0} \subset \Theta$ and $\Theta \cap \mathring{D^3} = \Theta_0$, such that $$\Theta \backslash \Theta_2 \subset \{ x_1^2 + x_2 ^2 > a_0^2 \}.$$

\noindent
Step 2: After that, we will verify $\Theta$ suits our requests. Consider a perturbation $\Sigma_{a_0}$ with $\Sigma_{a_0} (S_{a_0}) \subset \Theta$, and $H(\Sigma_{a_0}) \ge 1$, and we will prove $\Sigma_{a_0} = id$.

First, we assert that $$\Sigma(S_{a_0}) \cap D^3 \subset \{ x_1^2 + x_2^2 \ge a_0^2 \}$$
To prove this assertion, define
\begin{align*}
&B'_0 = \{ x_1^2 + x_2^2 = a_0^2\} \quad B'_1 = \{ x_1^2 + x_2^2 = \frac{1}{4}\}\\
&\Omega' = \{ \frac{1}{4} < x_1 ^2 + x_2^2 < a_0^2\}\\
&F'_t = \{ x_1 ^2 + x_2^2 = (a_0 - t)^2\}, 0< t < a_0 - \frac{1}{2}.
\end{align*}
It is apparent that $H(F'_t) < 1, \forall t$, so thay can be seen as our slice. Then, consider
$$\Sigma_1 = \Sigma(S_{a_0}) \cap D^3 \cap \{ x_1 ^2 + x_2^2 < a_0^2\}.$$

Since $\Sigma$ is an embedded map, it is apparent that $\partial \Sigma_1 \subset \{ x_1 ^2 + x_2^2 = a_0^2\}$, hence, if $\Sigma_1 \neq \varnothing$, we can choose a connected component of $\Sigma_1$, denoted as $\Sigma_2$, such that $\exists p' \in \Sigma_2$, satisfying $$Op' \cap \Sigma = \varnothing$$
where $Op'$ represent this open line segment. This shows the normal on $\Sigma_2$ is inward.

Hence, we choose $\Omega'$ as the trap, $\{ F'_t \}$ as the slice. Then from trap-slice lemma, we have the contradiction, showing actually $\Sigma_1 = \varnothing$.

 \begin{figure}[!h]
   \setlength{\belowcaptionskip}{0.1cm}
   \centering
   \includegraphics[width=0.56\textwidth]{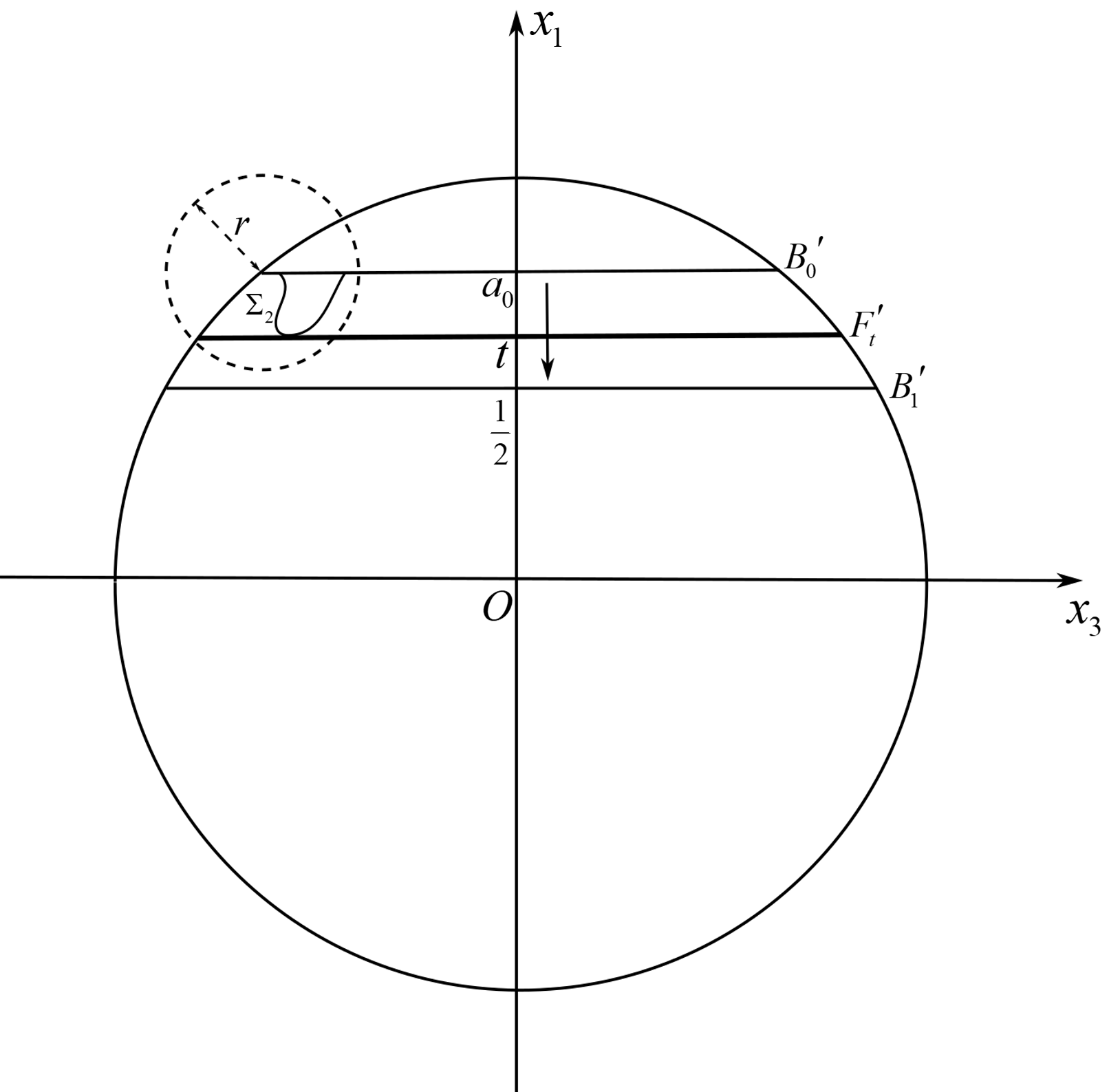}
   \label{fig5}
 \end{figure}

\noindent
Step 3: After that, we will prove $\Sigma(S_{a_0}) \cap D^3 \subset \overline{\Theta}_1$, and if not, define $$\Sigma'_1 = (\Sigma(S_{a_0}) \cap \Theta_2 \cap D^3) \backslash \overline{\Theta}_1 .$$
It is apparent that $\partial \Sigma_1' \subset \widetilde{C}(a_0, t_1)$. From Step 2, we also have $\Sigma_1' \subset S_{a_0}$. Since $\Theta_2$ has two connected components, denoted $\Theta^-$ the one with $x_3 < 0$, and might as well, we can assume that $\Sigma_1' \subset \Theta^-$ without loss of generality.

Define $E_t$ the translation of $B_1$ toward the vector $\tilde{\pmb{r}}_t = (0,0,t)(0 < t < r)$. Consider $$t' = \sup \{t \in (0,r) | E_t \cap \Sigma_1' \neq \varnothing \}$$
and it is apparent that $\Sigma_1'$ is tangent with $E_{t'}$ at some inner points.

Since $H(E_{t'}) = H^{a_0}(t_1) <1$, from Tangency Principle, we know $E_{t'} \cap \Sigma_1'$ is both open and close in $\Sigma_1'$, which will result in a contradiction.

 \begin{figure}[!h]
   \setlength{\belowcaptionskip}{0.1cm}
   \centering
   \includegraphics[width=0.52\textwidth]{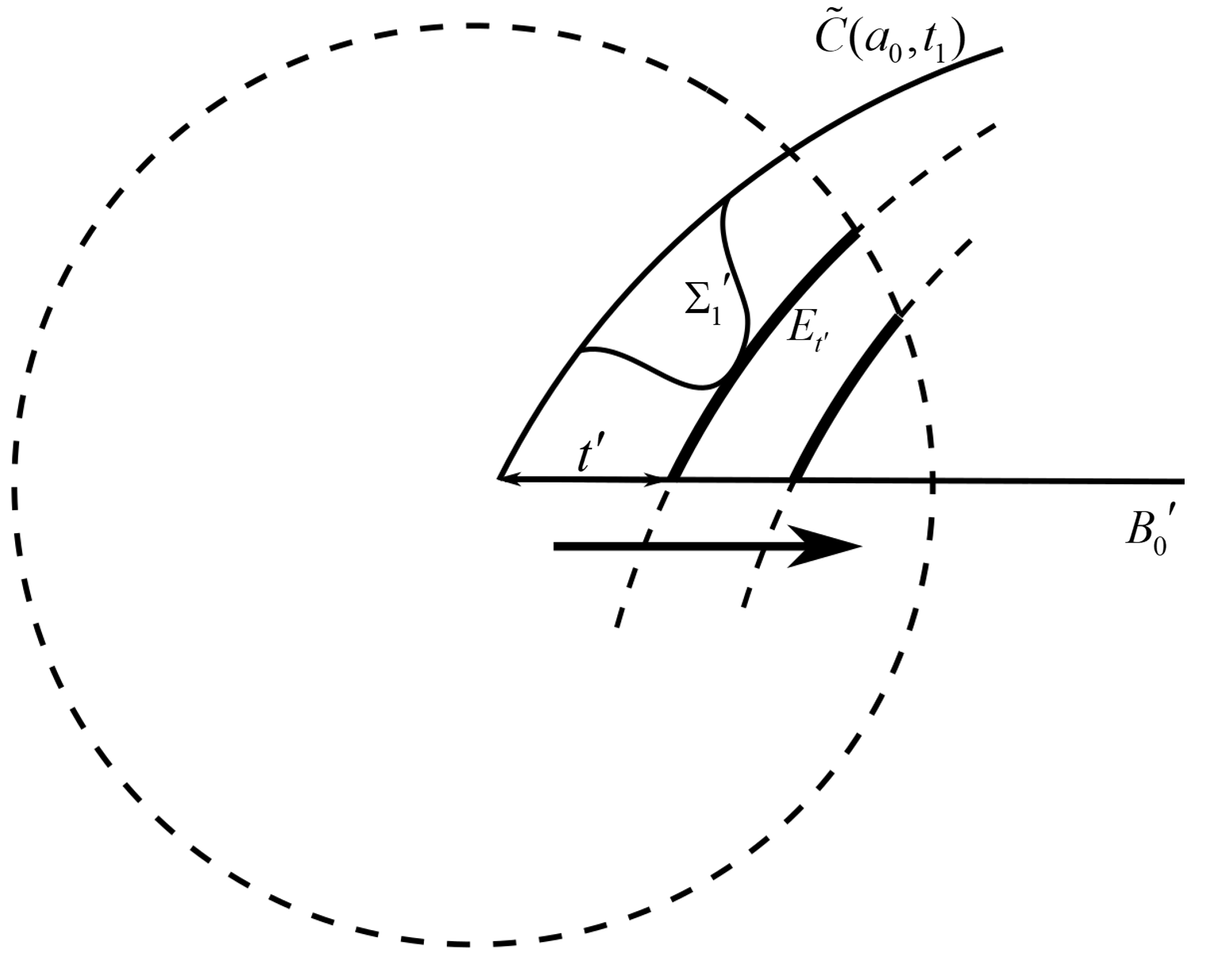}
   \label{fig6}
 \end{figure}

\noindent
Step 4: Now we will finish the proof of this theorem. If $\Sigma_{a_0} \neq id$, then select a connected component of $\Sigma_{a_0}(S_{a_0}) \cap \mathring{D^3}$, denoted as $\Sigma_*$, and from steps above, we know $\Sigma_* \subset \overline{\Theta}_1$. We can also assume that the normal of $\Sigma_*$ is inward by similar method in step 2.

It is apparent that $\exists t_1'' \in (t_1, 1)$, such that $$\widetilde{C}(a_0, t_1'') \cap \Sigma_* \neq \varnothing.$$
Select a connected component of $(\cup_{t \in [t_1, t_1'']}\widetilde{C}(a_0, t)) \cap \Sigma_*$, denoted as $\Sigma_{**}$. Since Remark~\ref{htmono}, we can easily get the proposition similar to Proposition~\ref{partial} on $\widetilde{C}(a_0, t_1'')$, which shows that $$\partial \Sigma_{**} \cap \partial S_{a_0} = \varnothing.$$
Then define $$t_2'' = \inf \{ t \in [t_1, t_1'']: \widetilde{C}(a_0, t) \cap \Sigma_{**} \neq \varnothing\}.$$
So $\Sigma_{**}$ must be tangent with $\widetilde{C}(a_0, t_2'')$ at some inner points. Also, $H^{a_0}(t_2'') < 1$, hence, similarly to the proof of Theorem~\ref{localH-}, we get the contradiction.

 \begin{figure}[!h]
   \setlength{\belowcaptionskip}{0.1cm}
   \centering
   \includegraphics[width=0.56\textwidth]{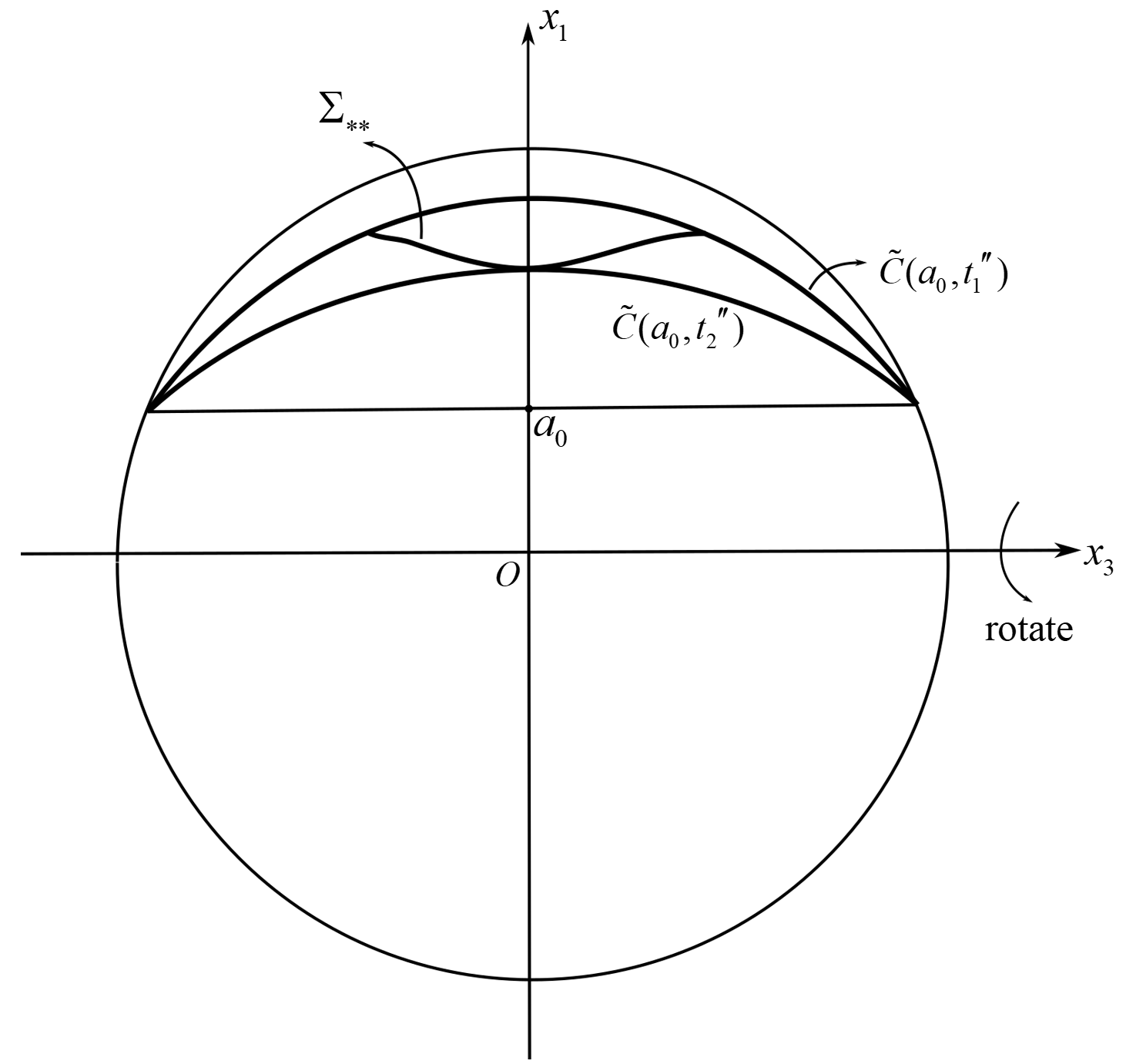}
   \caption*{Selection of $t_2''$ as the profile in $P_2$}
   \label{fig7}
 \end{figure}

In summary, we finish the proof.
\end{proof}

\begin{remark}
From those two theorems above, it can be noticed that when $a = a_0$, it has local $H^+$ rigidity, but actually $a_0$ does not have local $H^-$ rigidity. The reason is that $H^{a_0}(t)$ has a maximum when $t = 1$ which neither increase nor decrease. We will discuss it more carefully in the next part.
\end{remark}

\section{The round-corner lemma}
To show the non-rigidity part in our theorems, the basic idea is to gluing certain pieces of CMC surfaces with desired mean curvature functions and smoothing them at the intersection points(lines).


\begin{theorem}\label{round-corner-lemma}[The round-corner lemma]
Denote $\pmb{r}_1(s), \pmb{r}_2(s)$ as two regular parameter curves in $P_2$, with arc-length parameter $s$, and they have no intersection with themselves in their domain of definition. Suppose they transversely intersect at $p = \pmb{r}_1(0) = \pmb{r}_2(0)$ (which is assumed to be the only intersect point), and they are at the same side of $x_1$ axis, namely, $x_1 >0$. Given $\pmb{T}_i = \pmb{r}'_i / ||\pmb{r}'_i||$ the unit tangent vector field, and $\pmb{n}_i(i=1,2)$ the related left hand unit normal field of $\pmb{r}_i$ in $P_2$. Might as well, assume $\pmb{T}_2 \cdot \pmb{n}_{1} > 0$, which means $\pmb{r}_1$ should turn right to turn to $\pmb{r}_2$ at $s = 0$.\\
In $\R^3$, rotate each $\pmb{r}_i$ around $x_3$ axis to create the surface $R_i$, and $\pmb{n}_i$ naturally generates a normal field of $R_i$, still denoted as $\pmb{n}_i$. Define $H_i(s)$ the mean curvature of $C_i$ at $\pmb{r}_i(s) \in R_i$, from the normal $\pmb{n}_i$.\\
For each $ \{ \pmb{r}_i \} (i=1,2)$ that suit all the definitions and requirements above, and for $\forall \delta > 0$, there exists a round corner $\pmb{r}$ with a positive $0 < \epsilon \ll 1$, such that:
\begin{enumerate}
  \item $\pmb{r}(s)$ is a smooth regular parameter curve in $P_2$(but now $s$ may not be the length of curve), satisfying
\begin{align*}
\pmb{r}(s) =
\begin{cases}
\pmb{r}_1(s)& \text{$s \le - \epsilon$}\\
\pmb{r}_2(s)& \text{$s \ge \epsilon$}
\end{cases}
\end{align*}
but we do not request $s$ as the length of curve when $|s| \le \epsilon$.
  \item When $|s| \le \epsilon$, $\pmb{r}$ does not self intersect or intersect with other parts of itself in $B(p, \delta)$, and $d(\pmb{r}(s), p) < \delta$.
  \item Given $\pmb{T} = \pmb{r}' / ||\pmb{r}'||$ the unit tangent vector field, and it is apparent that its related left hand unit normal field, denoted as $\pmb{n}$, coinsides with $\pmb{n}_1$ or $\pmb{n}_2$ when $|s| > \epsilon$. Also define $R$ the surface of revolution of $\pmb{r}$ with the normal $\pmb{n}$ similar with above, and $H(s)$ the mean curvature of $R$ at $\pmb{r}(s)$. We have
\begin{align*}
\begin{cases}
H(s) \ge H_1(s)& \text{$ - \epsilon < s < 0$}\\
H(s) \ge H_2(s)& \text{$0 < s < \epsilon$}
\end{cases}
\end{align*}
\end{enumerate}
\end{theorem}

 \begin{figure}[!h]
   \setlength{\belowcaptionskip}{0.1cm}
   \centering
   \includegraphics[width=0.55\textwidth]{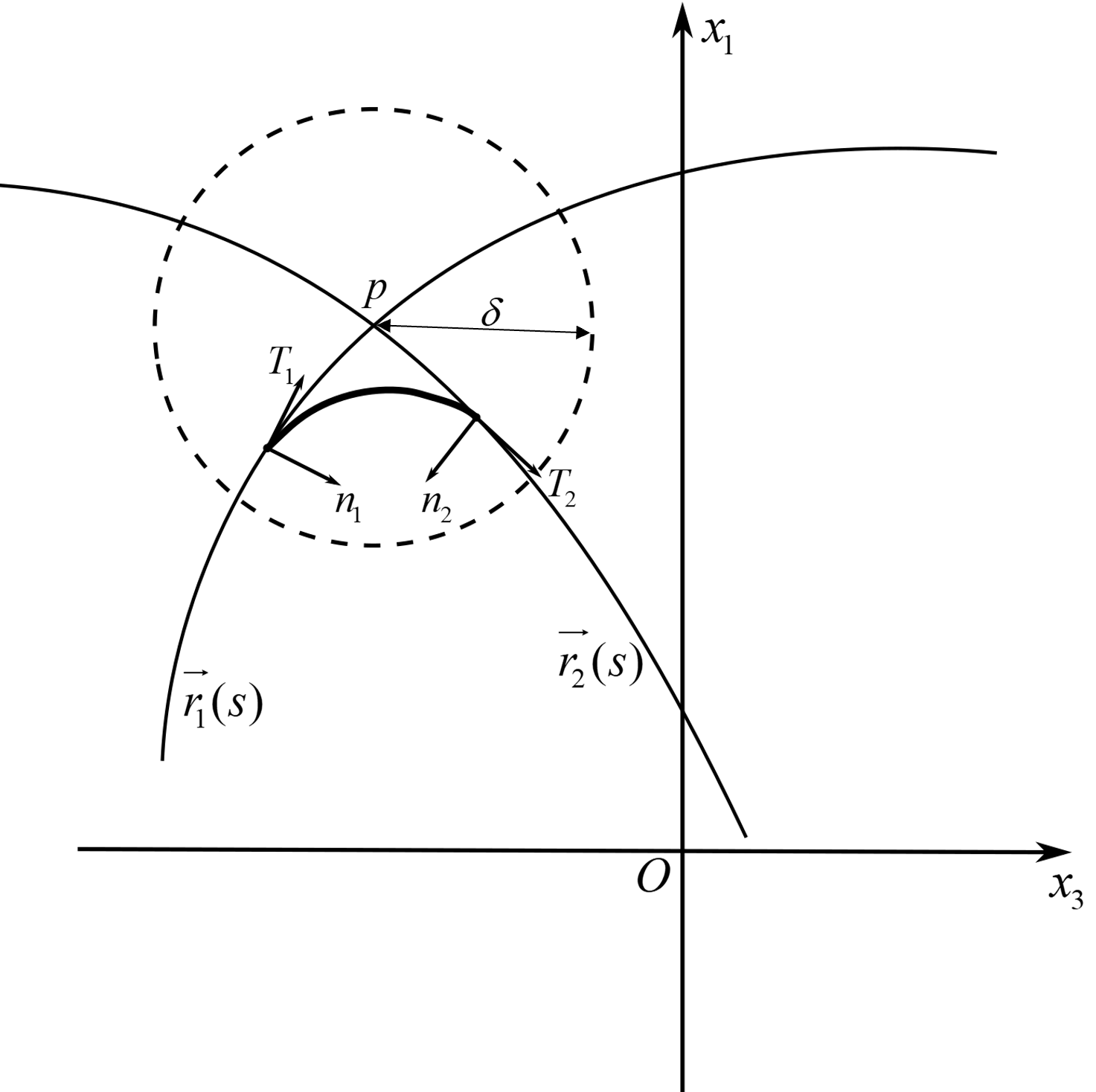}
	\caption*{The round-corner lemma}
   \label{fig8}
 \end{figure}

\begin{remark}\label{-ori}
It is apparent that if we first assume $\pmb{T}_2 \cdot \pmb{n}_{1} < 0$ rather than that in the lemma, then the lemma still keep true, with the change of requirement of the mean curvature, i.e.
\begin{align*}
\begin{cases}
H(s) \le H_1(s)& \text{$ - \epsilon < s < 0$}\\
H(s) \le H_2(s)& \text{$0 < s < \epsilon$}
\end{cases}
\end{align*}
\end{remark}

For the proof of the round corner lemma, we introduce three basic lemmas.

\begin{lemma}\label{intersect}
For $\forall D > 0, \exists d>0$, such that every two continuous curves $\pmb{l}_1(s),\pmb{l}_2(s) \subset \R^2$ will intersect at some $p^* \in B(p,D)$, if they always satisfy $$|| \pmb{l}_i(s) - \pmb{r}_i(s) || < d \quad i \in \{ 1,2 \}$$
\end{lemma}

\begin{proof}
We can find a disk $B$ that contains $p$, and make $d$ small enough that we can assume that $\pmb{l}_1(s) (s_1<s<s_2)$ and $\pmb{l}_2(s) (s_3<s<s_4)$ are also contained in $B$, with $\pmb{l}_1(s_1),\pmb{l}_2(s_3),\pmb{l}_1(s_2),\pmb{l}_2(s_4)$ arranged on its boundary in order. Then it is apparent that these two curve segments will intersect since Jordan Curve Theorem.
\end{proof}

\begin{lemma}\label{hikaku}
Using the symbols in round corner lemma, denote $\kappa_i$ as the curvature of $\pmb{r}_i$ about the normal $\pmb{n}_i$; $y_i$ as the ordinate of $\pmb{r}_i$; and $\theta_i = \arg \pmb{T}_i$. Suppose that $d_i(s) >0 , \forall s$, then there is always  $H_1(s) \ge H_2(s)$ if
\begin{align*}
\kappa_1(s)-\kappa_2(s) \ge \frac{|d_1(s)-d_2(s)|}{\min \{ d_1(s),d_2(s) \} ^2}+ \frac{|\theta_1(s)-\theta_2(s)|}{\min \{ d_1(s),d_2(s) \} }.
\end{align*}
\end{lemma}

\begin{proof}
Only need to observe that
\begin{align*}
2H_i(s)=\kappa_i(s)+ \frac{\cos \theta_i(s)}{d_i(s)}
\end{align*}
So we can estimate that
\begin{align*}
|\frac{\cos \theta_1(s)}{d_1(s)}-\frac{\cos \theta_2(s)}{d_2(s)}| & \le \frac{|d_1(s) \cos \theta_2(s)-d_2(s) \cos \theta_1(s)|}{d_1(s) d_2(s)}\\
& \le \frac{|d_1(s) \cos \theta_2(s)-d_1(s) \cos \theta_1(s)|+|d_1(s) \cos \theta_1(s)-d_2(s) \cos \theta_1(s)|}{d_1(s) d_2(s)}\\
& \le \frac{|\theta_1(s) - \theta_2(s)|}{d_2(s)} + \frac{|d_1(s) - d_2(s)|}{d_1(s) d_2(s)}\\
& \le \frac{|\theta_1(s)-\theta_2(s)|}{\min \{ d_1(s),d_2(s) \} } + \frac{|d_1(s)-d_2(s)|}{\min \{ d_1(s),d_2(s) \} ^2}.
\end{align*}
Then the lemma is trivial by this estimation and the given condition.
\end{proof}

\begin{lemma}\label{angle}
Consider any minor arc of circle $\overset{\frown}{A_1A_2} \subset \R^2$, whose length is denoted $L$. Also, request the central angle of $\overset{\frown}{A_1A_2}$ are all in $[\alpha_1, \alpha_2] \subset (0,\pi)$.

If $B_1,B_2$ are any two points with $d(A_i,B_i) \le \rho$, then we have $$\lim_{\frac{\rho}{L} \to 0^+} (\arg \overrightarrow{B_1B_2}  - \arg \overrightarrow{A_1A_2}) = 0 .$$
\end{lemma}

\begin{proof}
Without loss of generality, we can assume $L = 1$. Then we only need to prove $$\lim_{\rho \to 0^+} (\arg \overrightarrow{B_1B_2} -\arg \overrightarrow{A_1A_2}) = 0.$$
Define $\theta_{12}$ as the central angle of $\overset{\frown}{A_1A_2}$, then it is trivial to verify that $$|A_1A_2| = \frac{2}{\theta_{12}}\sin\frac{\theta_{12}}{2}.$$
Since $\theta_{12} \in [\alpha_1, \alpha_2]$, it is apparent that $$|A_1A_2| \ge \frac{2}{\alpha_2}\sin \frac{\alpha_1}{2} > 0.$$
Also, it is apparent that $$|\overrightarrow{B_1B_2} - \overrightarrow{A_1A_2}| \le |\overrightarrow{A_1B_1}| + |\overrightarrow{A_2B_2}| \le 2\rho.$$
Hence, this lemma is trivial from the Law of Cosines, when $\rho \to 0^+$.
\end{proof}

Now we will prove the round-corner lemma. First we define $s(x)$ as
\begin{align*}
s(x) =
\begin{cases}
0& \text{ $ x \le 0 $ } \\
e^{-\frac{1}{x}}& \text{ $ x > 0 $ }
\end{cases}
\end{align*}
Then, define $\mu_{\lambda}(x)$ for $\lambda > 0$ as $$ \mu_{\lambda}(x) = \frac{s(x)}{s(x) + s(\lambda-x)}$$
It can be noticed that $\mu_\lambda \in C^\infty (\R)$, $\mu_\lambda$ increases, and $0 \le \mu_\lambda \le 1$.

Without loss of generality, we can assume that $\delta  < y_p$, so we can find $d_0 \in (0, y_p - \delta)$ and $\epsilon_1 > 0$, such that $\pmb{r}_1(s)$ and $\pmb{r}_2(s)$ $(|s| \le \epsilon_1)$ lie in $B(p,\delta)$ and intersect at only one point $p$.

Consider $\pmb{T}_0 = \pmb{T_1} + \pmb{T}_2$ and denote $\theta = \arg \pmb{T_0}(0)$. Since $\pmb{T}_1$ should turn right to $\pmb{T}_2$, define $$\theta^* = \arg \pmb{T}_1(0) - \theta \in (0, \frac{\pi}{2})$$
and select one
\begin{align}\label{beta}
\beta \in (0, \min \{ \frac{\theta^*}{4}, \frac{\pi}{4} - \frac{\theta^*}{2} \} ).
\end{align}
Then, it is trivial from the continuity of $\arg \pmb{T}_i$ that we can select $\epsilon_2 \in (0, \epsilon_1)$ and $\delta' < \delta / 2$, such that $d(\pmb{r}_i(s),p) < \delta'$, with
\begin{align}\label{beta1}
\arg \pmb{T}_1 &\in [ \theta+ \theta^* - \beta,\theta+\theta^* + \beta ] \quad \forall s \in [-\epsilon_2,0],
\end{align}
\begin{align}\label{beta2}
\arg \pmb{T}_2 &\in [ \theta-\theta^* - \beta,\theta-\theta^* + \beta ] \quad \forall s \in [0,\epsilon_2].
\end{align}
Hence, we can also find $\alpha \in (0, \pi / 4)$, such that
\begin{align*}
\arg \pmb{T}_1 \in [ \theta+\alpha,\theta+\frac{\pi}{2}-\alpha ],\\
\arg \pmb{T}_2 \in [ \theta-\frac{\pi}{2}+\alpha,\theta-\alpha ].
\end{align*}
Fix $\theta, \alpha$ and $\beta$.

 \begin{figure}[!h]
   \setlength{\belowcaptionskip}{0.1cm}
   \centering
   \includegraphics[width=0.4\textwidth]{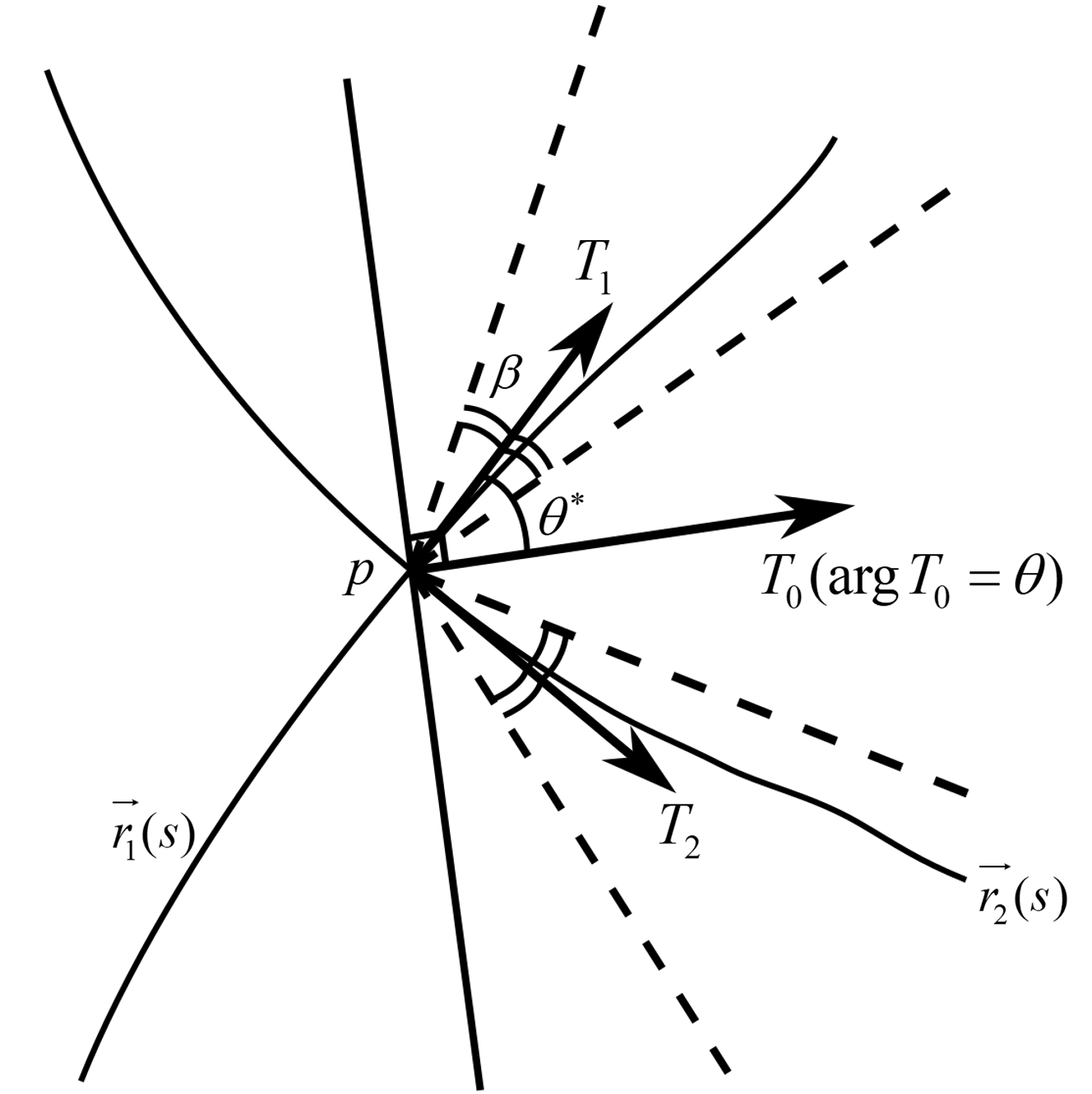}
	\caption*{Selection of $\beta$, the four angles marking two arcs are all $\beta$}
   \label{fig9}
 \end{figure}

In order to create the smooth polishing curve, first we will half-polish them by using two half-polishing curves, which will be expected to get closer to the same arc of circle smoothly. First, by using the Fundamental Theorem of Curve Theory, we generally define the smooth polishing curve $\pmb{P}_1(s_0,\lambda ,K, s)$, with $s$ as the length of curve, $s_0,\lambda,K$ as the undetermined parameters, requesting $$\pmb{P}_1(s_0,\lambda,K,s_0) = \pmb{r}_1(s_0).$$
And we set its curvature as
$$\widetilde{\kappa}_1(s) = \kappa_1(s)+(K-\kappa_1(s))\mu_\lambda(s-s_0)$$
So it is apparent that $$\pmb{P}_1(s_0,\lambda,K,s)=\pmb{r}_1(s) \quad \forall s \le s_0$$
and when $s \ge s_0+\lambda$, $\pmb{P}_1$ becomes a circle. We use $O_1(s_0,\lambda,K)$ to represent its center.

Conversely, define $\pmb{P}_2(s_0,\lambda ,K, s)$, with $s$ as the length of curve, $s_0,\lambda,K$ as the undetermined parameters same as $\pmb{P}_1$, requesting $$\pmb{P}_2(s_0,\lambda,K,s_0) = \pmb{r}_2(s_0).$$
Set its curvature as 
$$\widetilde{\kappa}_2(s) = \kappa_2(s)+(K-\kappa_2(s))\mu_\lambda(s_0-s).$$
So we have $$\pmb{P}_2(s_0,\lambda,K,s)=\pmb{r}_1(s) \quad \forall s \ge s_0$$
and when $s \le s_0 - \lambda$, $\pmb{P}_2$ is a circle, whose center is denoted as $O_2(s_0,\lambda,K)$.\\

We will try to connect $\pmb{P}_1$ and $\pmb{P}_2$ by an arc. We should always request $\lambda,K$ to be positive. Then, when we fix $\lambda$ and $K$, we have two continuous curves $O_i(s,\lambda,K)(i=1,2)$, and it is apparent that $$||O_i(s) - \pmb{r}_i(s)|| \le  \lambda+\frac{1}{K}.$$

Using Lemma~\ref{intersect}, if $\lambda + 1/K$ is small enough, we can select the intersection point of $O_1, O_2$ as $O(\lambda, K)$, satisfying 
$$d(O(\lambda,K),p) \le \delta',$$
and select a corresponding $s_i(\lambda,K) \in [-\epsilon_2, \epsilon_2]$, such that $$O_i(s_i(\lambda,K),\lambda,K) = O(\lambda,K).$$
Since we have used $\lambda, K$ to give one $s_i$, the undetermined parameters left are only $\lambda$ and $K$. We define 
$$\phi_i(\lambda,K,s) = \arg \pmb{P}_i'(s_i,\lambda,K,s).$$
Since $\pmb{P}_1$ and $\pmb{P}_2$ have the same center as $O(\lambda, K)$, consider the arc of circle from $\pmb{P}_1(s_1+\lambda)$ to $\pmb{P}_2(s_2-\lambda)$(though we do not know whether it will cause self intersection right now), and we define $L$ as the length of this arc. It is apparent that 
$$L \le \frac{2\pi}{K}.$$
Then $\pmb{P}_1$, this connected arc, and $\pmb{P}_2$ will joint together into a round corner, defined as $\pmb{r}$.

After that, we will guarantee that this round corner will not intersect itself. For these requests, we need to estimate the variation of $\phi_i$.

If $K > \sup_{|s| \le \epsilon} |\kappa_1(s)|$, then for $\forall \lambda' \in(0, \lambda)$, we have
\begin{align}
| \phi_1(s_1) - \phi_1(s_1 + \lambda') | &\le \int_{s_1}^{s_1+\lambda'} |\kappa_1(s)+(K-\kappa_1(s))\mu_\lambda(s-s_1)|ds\\
& \le \int_{s_1}^{s_1+\lambda'} |\kappa_1(s)+(K-\kappa_1(s))\mu_\lambda(s-s_1)|ds \\
& \le \int_{s_1}^{s_1+\lambda'}Kds + \int_{s_1}^{s_1+\lambda'}|\kappa_1(s)|ds \\
& \le \int_{s_1}^{s_1+\lambda'}2Kds = 2K\lambda'. \label{2krou}
\end{align}
Similarly, we have
\begin{align}\label{2krou2}
|\phi_2(s_2 - \lambda') - \phi_2(s_2)| \le 2K\lambda' \quad \forall \lambda' \in(0, \lambda)
\end{align}
Hence, if $2\lambda K < \alpha$, then $\forall \lambda' \in(0, \lambda]$, we have

$$\phi_1(s_1 + \lambda') \in (\theta, \theta + \frac{\pi}{2}) \quad \phi_2(s_2 - \lambda') \in (\theta - \frac{\pi}{2},\theta).$$

Then consider their position on the circle. Since $\pmb{n}_i$ is left hand, when $\pmb{P}_i$ becomes the circle, they must rotate clockwise, hence, $\phi_1(s_1 + \lambda)$ will decrease to $\phi_2(s_2 - \lambda)$, showing that it is the minor arc.

 \begin{figure}[!h]
   \setlength{\belowcaptionskip}{0.1cm}
   \centering
   \includegraphics[width=0.4\textwidth]{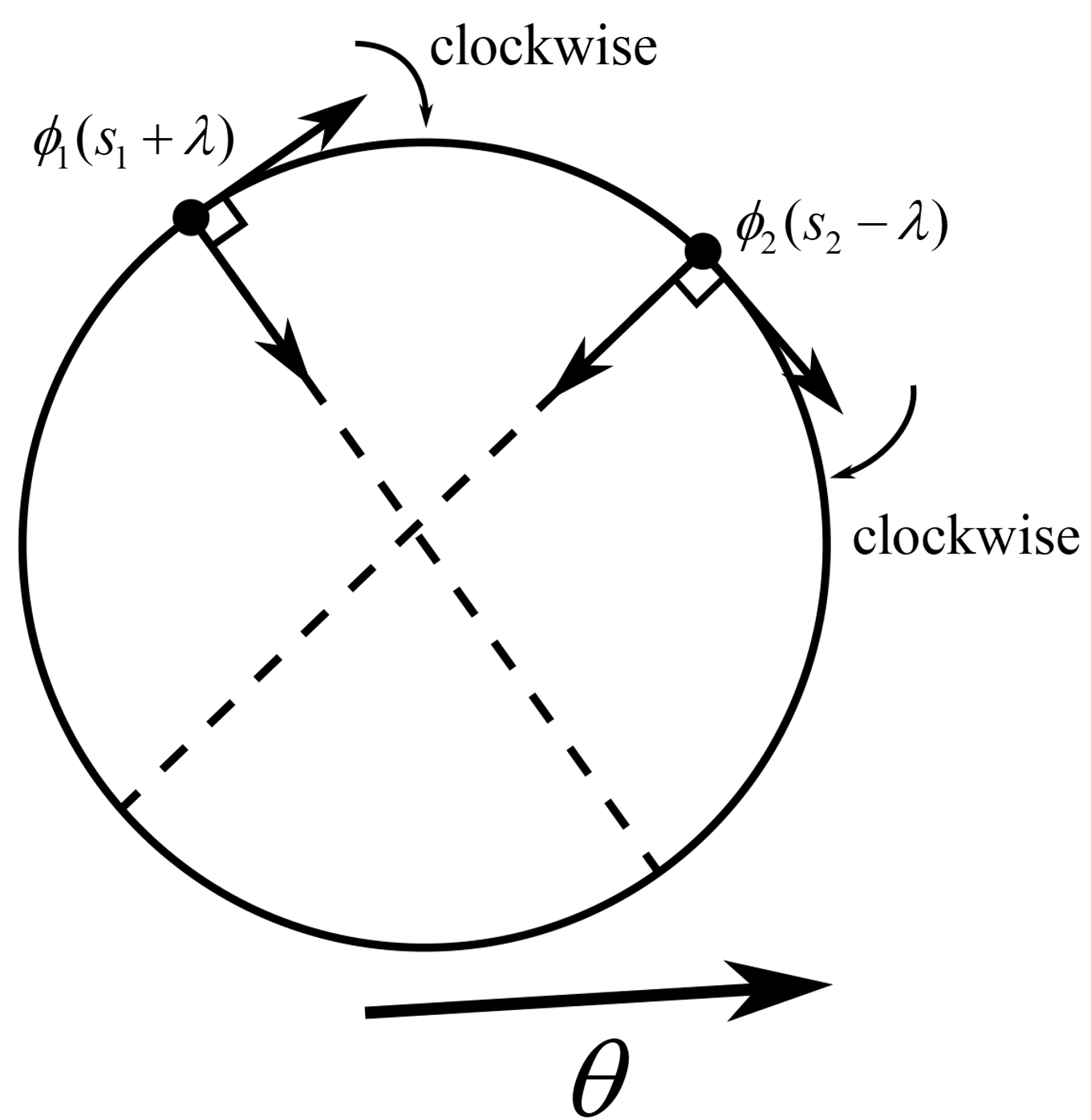}
	\caption*{$\phi_1(s_1 + \lambda)$ will decrease to $\phi_2(s_2 - \lambda)$, showing the corner must be a minor arc}
   \label{fig10}
 \end{figure}

Also, it is apparent that $$\pmb{r}' (s) \cdot \pmb{T}_0 (0) > 0 \quad \forall s \in (-\epsilon_2, \epsilon_2)$$
and this means that $\pmb{r}(s)$ has no self intersection.

We can also guarantee that $s_1(\lambda,K)\le 0$ and $s_2(\lambda,K) \ge 0$. Actually, $s_1$ and $s_2$ can not be $0$ simultaneously, since $\pmb{r}'(s) \cdot \pmb{T}_0(0) > 0$ and $\pmb{T}_2 \cdot \pmb{n}_1 > 0$. It is also impossible that $s_1 \ge 0,s_2 \le 0$ for the same reason.

For other conditions, we need a more accurate estimation of the angle.

If $2\lambda K < \beta$, then it is apparent from~\ref{beta1}, \ref{beta2}, ~\ref{2krou} and~\ref{2krou2} that $$|\phi_1(s_1+\lambda) -(\theta + \theta^*)| \le 2\beta \quad |\phi_2(s_2 - \lambda) -(\theta - \theta^*)| \le 2\beta.$$
Also, from the basic geometric feature of the arc, we have
\begin{align*}
\arg[\pmb{P}_2(s_2 - \lambda)-\pmb{P}_1(s_1 + \lambda)] =  \frac{\phi_2(s_2 - \lambda)+ \phi_1(s_1 + \lambda)}{2}.
\end{align*}
Hence, we can directly estimate that
\begin{align*}
|\arg[\pmb{P}_2(s_2 - \lambda)-\pmb{P}_1(s_1 + \lambda)] - \theta| \le 2\beta.
\end{align*}

After that, notice that
$$L = \frac{|\phi_2(s_2 - \lambda)- \phi_1(s_1 + \lambda))|}{K} \ge \frac{2\theta^* - 4\beta}{K}.$$
Define 
$$\rho_1 = d(\pmb{P}_1(s_1), \pmb{P}_1(s_1 + \lambda)) \quad \rho_2 = d(\pmb{P}_2(s_2), \pmb{P}_2(s_2 - \lambda)).$$
From~\ref{beta}, we have 
$$ \frac{L}{\rho_i} \ge \frac{2\theta^* - 4\beta}{\lambda K} \ge \frac{\theta^*}{\lambda K}.$$
Hence, if $\lambda K$ is small enough, from Lemma~\ref{angle}, we can assume that $|\arg[\pmb{P}_2(s_2)-\pmb{P}_1(s_1)]-\arg[\pmb{P}_2(s_2 - \lambda)-\pmb{P}_1(s_1 + \lambda)]|$ is small enough, such as 
$$|\arg[\pmb{P}_2(s_2)-\pmb{P}_1(s_1)] - \theta| \le 3\beta.$$

Therefore, from $3\beta < \theta^* - \beta$, it is impossible that $s_1,s_2 \ge 0$ or $s_1, s_2 \le 0$, because if not, since $s_1$ and $s_2$ can not be both equal to $0$, from trivial geometric fact(see from the figure below), there must be
$$|\arg[\pmb{P}_2(s_2)-\pmb{P}_1(s_1)] - \theta| > \theta^* - \beta,$$
which is a contradiction.

 \begin{figure}[!h]
   \setlength{\belowcaptionskip}{0.1cm}
   \centering
   \includegraphics[width=0.45\textwidth]{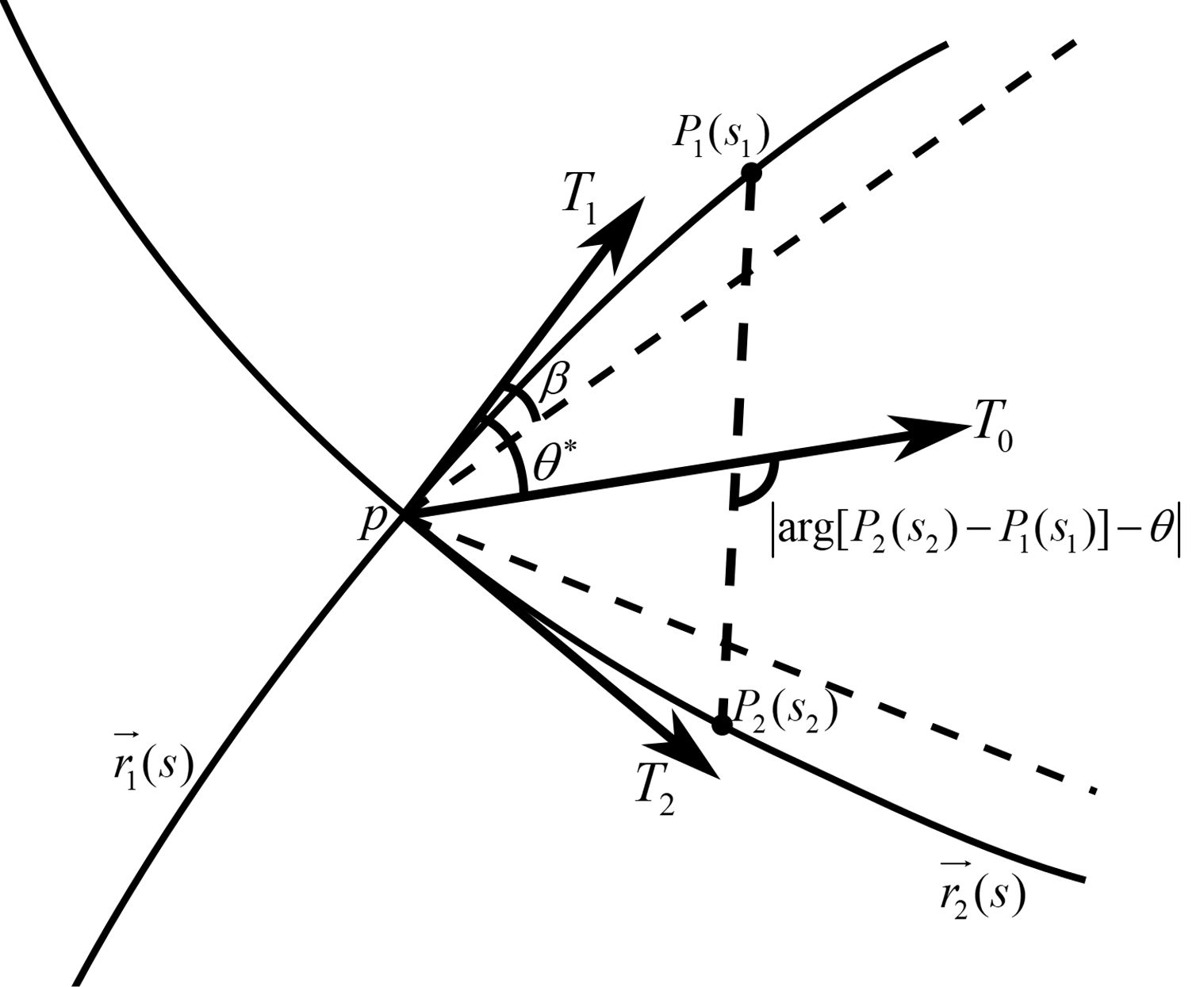}
	\caption*{Contradiction if $s_1, s_2 >0$, the other condition is similar}
   \label{fig11}
 \end{figure}

Finally, we try to control the mean curvature of the round corner. The arc part has mean curvature $$H(s) = \frac{1}{2} (K+\frac{\cos \theta(s)}{d(s)}) \ge \frac{1}{2}(K-\frac{1}{d(s)})$$
where $d(s)$ respects the distance from $\pmb{r}(s)$ to $x_3$ axis. So we only need
\begin{align*}
K> 2\sup_{|s| \le \epsilon, i = 1,2} |\kappa_i(s)|+\frac{1}{d_0}.
\end{align*}
As for the half-polishing part, according to Lemma~\ref{hikaku}, we only need to keep $$(K-\kappa_1(s))\mu_\lambda(s-s_1)\ge \frac{|d_1(s)-d_c(s)|}{(\min \{ d_1(s),d_c(s) \} )^2}+ \frac{|\theta_1(s)-\theta_c(s)|}{\min \{ d_1(s),d_c(s) \} }   (s_1 \le s \le s_1 + \lambda).$$
where $d_c$ represents the distance from $\pmb{P}_1(s)$ to $x_3$ axis, and $d_1$ from $\pmb{r}_1(s)$ to $x_3$ axis, and the similar request of $\pmb{P}_2$.

This time, we let $K$ be large enough such that $\exists a, b > 0, a < b < 2a$, such that
$$K-\kappa_i(s) \in [a,b] \quad \forall s \in [-\epsilon,\epsilon], i=1,2.$$
Then notice that

\begin{align*}
|( \theta_1(s)-\theta_c(s))| &\le \int_{s_1}^s (K-\kappa_1(s))\mu_\lambda (s-s_1) ds \\
& \le \int_{s_1}^s b\mu_\lambda (s-s_1) ds \le b(s-s_1)\mu_\lambda (s-s_1).
\end{align*}

\begin{align*}
|d_1(s)-d_c(s)|&=|\int_{s_1}^s (\cos \theta_1(s) - \cos \theta_c(s))ds|\\
& \le \int_{s_1}^s |( \theta_1(s)-\theta_c(s))|ds \\
&\le \int_{s_1}^s b(s-s_1)\mu_\lambda (s-s_1)ds\\
&\le b\mu_\lambda (s-s_1)\int_{s_1}^s (s-s_1)ds = \frac{b}{2} (s-s_1)^2 \mu_\lambda (s-s_1).
\end{align*}

So we only need $$ a\mu_\lambda (s-s_1) \ge \frac{b(s-s_1)^2 \mu_\lambda (s-s_1)}{2d_0^2} +\frac{b(s-s_1)\mu_\lambda (s-s_1)}{d_0}.$$

And this will be guaranteed if
$$\frac{\lambda ^2}{d_0^2} + \frac{2\lambda}{d_0} \le 1,$$
which will keep true when $\lambda < d_0 / 4$.

Actually, for $\pmb{P}_2$, this restriction is still valid from the similar estimation.

From all the calculation and analysis above, it can be noticed that we only need to keep $\lambda \ll 1, K \gg 1$ and $\lambda K \ll 1$, and this will be easily guaranteed. Finally, we complete the whole proof by select $\epsilon = \epsilon_2$, $K$ sufficiently large and $\lambda K$ sufficiently small.

\section{Non-rigidity and perturbations}

The round-corner lemma and Remark~\ref{-ori} make it easy to construct a series of non-trivial perturbations, hence establish the non-rigidity in four different situations as below. It should be noticed that we will then only consider the generatrices in $P_2$, denoted by $\pmb{r}$ consistently, and $R, \pmb{n}, \pmb{T}$ are also consistent with those in the proof of round-corner lemma.

We will first consider the global rigidity.

\begin{theorem}\label{dePrigi}
For $\forall a \in (0, 1)$, $S_a$ does not have $H^-$ rigidity
\end{theorem}

\begin{proof}
We only need to construct $\pmb{r}$ as a properly perturbed generatrix from $S^1 \subset P_2$ to create surface $R$ such that $H(R) \ge  1$. Denote $a' = (1+a) / 2$ and $\pmb{r}_1$ the major arc of the unit circle under $x_1 = a'$, i.e.
\begin{align*}
\pmb{r}_1 = \{(x_3,x_1) \in P_2 | x_3^2 + x_1^2 = 1, x_1 \le a' \}
\end{align*}
Let $s$ the length of curve, and $\pmb{r}_1(0) = (-\sqrt{1-a'^2}, a')$, with $x_3$ the abscissa. Then we assign $\pmb{T}_1(0) = (a', \sqrt{1-a'^2})$ (this orientation means $s \le 0$ for $\pmb{r}_1$) and $\pmb{n}_1(0) = (\sqrt{1-a'^2}, -a')$.

Similarly, denote $\pmb{r}_2$ the symmetry of $\pmb{r}_1$ by $x_1 = a'$, i.e.
\begin{align*}
\pmb{r}_2 = \{(x_3,x_1) \in P_2 | x_3^2 + (x_1 - 2a')^2 = 1, x_1 \ge a' \}
\end{align*}
Also, let $s$ the length of curve, and $\pmb{r}_2(0) = (-\sqrt{1-a'^2}, a')$. Then denote $\pmb{T}_2(0) = (-a', \sqrt{1-a'^2})$ and $\pmb{n}_2(0) = (\sqrt{1-a'^2}, a')$.

Denote $p = \pmb{r}_i(0)$ and $\delta = (a' - a) / 2$. It is apparent that $\pmb{T}_2 \cdot \pmb{n}_1 < 0$, so by Remark~\ref{-ori}, we can select $\epsilon > 0$ and construct $\pmb{r}$ suiting all the requests in the round-corner lemma, with $H(s) \le 1, |s| < \epsilon$.

Also, we can construct this perturbation symmetrically about the $x_1$ axis, which will generate a complete perturbation of $S^1$. When $|s| < \epsilon$, $\pmb{r}$ will not intersect with other parts of $\pmb{r}$, so this perturbation can be seen as an embedded map.

Since $a' > a$ and $\delta < a' - a$, the surface of revolution can be seen as a perturbation of $S_a$, satisfying all of requests. Therefore, we finish the construction, which shows $S_a$ does not exist $H^-$ rigidity.
\end{proof}

\begin{theorem}\label{-Prigi}
Suppose $a \in (0, \sqrt{3} / 2)$, then $S_a$ does not have $H^+$ rigidity.
\end{theorem}

\begin{proof}
We still need to construct $\pmb{r}_2$ first ($\pmb{r}_1$ is naturally the original unit circle similarly to Theorem~\ref{dePrigi}), and this time the unduloid used in step 2 of the proof of Theorem~\ref{Prigi} will be chosen again.

Given $t \in (0, 1/2)$, consider the system $x_1 = u_t(x_3)$ which have been used once in section 2. It is known that $u_t$ increases when $0 < x_3 < 1$(since the half perimeter of $E_t$ is greater than $1$), and this time we will write the ODE of them as:
\begin{align*}
\frac{dx_1}{dx_3} = \sqrt{(\frac{x_1}{x_1^2 +t-t^2})^2 - 1}\\
x_1(0) = t.
\end{align*}

Now consider the first $P_t \in u_t$ with $x_1 (P_t) = a$, i.e.
$$x_3(P_t) = u_t^{-1}(a),$$
and it is easy to solve the ODE to get
\begin{align*}
x_3(P_t) &= \int_{t}^{a} \frac{1}{D(1, t, x_1)}dx_1 = \int_{0}^{a-t} \frac{1}{D(1, t, x + t)}dx
\end{align*}
It can be verified directly from~\ref{D} that $$\frac{\partial D(1,t,x+t)}{\partial t} < 0.$$
Hence from Lebesgue dominated convergence theorem, there is
\begin{align*}
\lim_{t \rightarrow 0^+}x_3(P_t) = \int_{0}^{a} \frac{1}{D(1, 0, x)}dx = 1 - \sqrt{1 - a^2} < \sqrt{1-a^2}.
\end{align*}
The last inequality sign is because $a < \sqrt{3}/{2}$.

Then fix $t$, and define $Q_t = u_t \cap S^1$ with $x_3(Q_t) > 0$, and it is apparent that $x_1(Q_t) > a$. Then, let $\pmb{r}_1(s)$ the unit circle with $x_1 \le x_1(Q_t)$, satisfying $\pmb{r}_1(0) = Q_t$, and $\pmb{r}_1$ rounds clockwise as $s$ increases. Also define $\pmb{r}_2$ the curve of $u_t$ with $x_3 \le x_3(Q_t)$, satisfying $\pmb{r}_2(0) = Q_t$, and $\pmb{r}_2$ runs in positive direction of $x_3$ as $s$ increases.

After that, it is easy to verify that $\{ \pmb{r}_i \}$ and their orientation suit all the requests in the round-corner lemma, with $\pmb{T}_2 \cdot \pmb{n}_1 > 0$. Let $\delta = (x_1(Q_t) - a)/2$, and from round-corner lemma, we contruct $\pmb{r}$ with $0 < \epsilon \ll 1$ near $Q_t$, which satisfies $H(s) \ge 1, |s| < \epsilon$. Since we can perturb the corner near the intersection symmetrically about $x_1$ axis, we also finish this proof.
\end{proof}

For the local rigidity, we need an easy proposition to help us select the original curve to be polished. With the parameters consistent with section 3, we define
$$\hat{c}(a,t) = c(1,t) \cap \{ |x_3| \le \sqrt{1-a^2} \}$$
and $\hat{C}(a,t)$ accordingly, and it is apparent that $\hat{C}(a,1) = S_a$.

\begin{proposition}\label{findpolish}
For $a \in (0,1)$, consider any open domain $\Theta \subset \R^3$ satisfying $\overline{S}_a \subset \Theta$.
Then $\exists \epsilon^{\Theta}_a \in (0, \epsilon''_a)$, such that
$$\hat{C}(a,t) \subset \Theta, \forall t \in [1- \epsilon^{\Theta}_a, 1 + \epsilon^{\Theta}_a].$$
\end{proposition}

This proposition is obvious from the continuity of $\hat{C}$ and the compactness of $S_a$.

\begin{theorem}\label{local1}
Suppose $a \in (0, a_0)$, then $S_a$ does not have local $H^+$ rigidity.
\end{theorem}

\begin{proof}
If $S_a$ has this rigidity, then there exists a domain $\Theta \subset \R^3$ such that $\overline{S}_a \subset \Theta$ and there is not non-trivial perturbation of $S_a$ satisfying $H \ge 1$.

From Proposition~\ref{findpolish}, we can select $t' \in ( 1-\epsilon^{\Theta}_{a}, 1)$, such that $\hat{C}(a, t') \subset \Theta$.

From the definition of $\hat{C}$, we know $H(\hat{C}) = 1$. From lemma~\ref{H=1}, there is $$x^*(a,1,t') > \sqrt{1-a^2}.$$
Hence, we have $$\hat{c}(a,t',\sqrt{1-a^2}) > a.$$

Therefore, we take 
$$\hat{x}_3 = \sup \{ x_3 > 0:\hat{c}(a,t',x_3) < \hat{c}(a,1,x_3)\}$$
so it is apparent that $\hat{x}_3 < \sqrt{1-a^2}$, and 
$$\hat{c}(a,t',\hat{x}_3) = \hat{c}(a,1,\hat{x}_3) \triangleq \hat{x}_1 > a.$$

In $P_2$, define two points $$Q_1(-\hat{x}_3, \hat{x}_1) \quad Q_2(\hat{x}_3, \hat{x}_1)$$
From the compactness of $\partial S_{a}$, there exists $\hat{\delta} \in (0, \hat{x}_1 -a)$, such that if we rotate $\cup_{i = 1,2} \overline{B(Q_i, \hat{\delta})}$ around $x_3$ axis to generate a closed domain in $\R^3$, denoted $\Theta'$, then $\Theta' \subset \Theta$.

After that, we can try to use round-corner lemma to finish the construction of the perturbations. Denote $\pmb{r}_1(s)$ the unit circle with $x_1 \le  a$, and when $s$ increases, $\pmb{r}_1$ rotates clockwise, and let $\pmb{r}_1(0) = Q_1$. Also define $\pmb{r}_2(s)$ the curve of $c_t$, with $x_3 = x_3(s)$ increasing. Also, let $\pmb{r}_2 (0) = Q_1$. Therefore, we only need to verify $\pmb{T}_2 \cdot \pmb{n}_1 > 0$, then the method of construction is similar to theorem~\ref{-Prigi}.\\
Now we try to prove $\pmb{T}_2 \cdot \pmb{n}_1 > 0$. Actually, it is equivalent to prove $$\frac{d\hat{c}(a,t')}{dx_3}|_{\hat{x}_3} > - \sqrt{(\frac{1}{a})^2 - 1}$$
Since $\hat{c}(a,t') < \hat{c}(1,1), |x_3|\le \sqrt{1-a^2}$, the $\ge$ can be directly guaranteed. If the equality holds, then we have $$D(1,t',\hat{x}_1) = D(1,1,\hat{x}_1) \Rightarrow t' = 1$$
which is apparently a conflict.

Therefore, we finish the proof.
\end{proof}

\begin{theorem}
Suppose $a \in (0, a_0]$, then $S_a$ does not have local $H^-$ rigidity.
\end{theorem}

\begin{proof}
Since Remark~\ref{monotonicity}, we only need to consider $a = a_0$. If $S_{a_0}$ has this rigidity, then there exists a domain $\Theta \subset \R^3$ such that $\overline{S}_{a_0} \subset \Theta$ and there is not non-trivial perturbation of $S_{a_0}$ satisfying $H \le 1$.

Since Proposition~\ref{findpolish}, we can select $t_0' \in (1, 1+ \epsilon^{\Theta}_{a_0})$, such that $\hat{C}(a_0, t'_0) \subset \Theta$. Since $x^*(a_0,1,t_0') < \sqrt{1-a_0^2}$ from lemma~\ref{H=1}, there is 
$$\hat{c}(a_0,t_0',\sqrt{1-a_0^2}) < a_0.$$

Hence, take $$\hat{x}'_3 = \sup \{ x_3 > 0|\hat{c}(a_0,t_0',x_3) > \hat{c}(1,1,x_3)\}.$$
So it is apparent that $\hat{x}_3' < \sqrt{1-a_0^2}$, and 
$$\hat{c}(a_0,t'_0,\hat{x}'_3) = \hat{c}(a_0,1,\hat{x}'_3) \triangleq \hat{x}'_1 > a_0.$$

Completely similar to the proof of theorem~\ref{local1}, what we only need to verify is $$\frac{d\hat{c}(a_0,t_0')}{dx_3}|_{\hat{x}'_3} < - \sqrt{(\frac{1}{a_0})^2 - 1}.$$
Also, the $\le$ has been guaranteed, and if the equality holds, we have $$D(1,t_0',\hat{x}'_1) = D(1,1,\hat{x}'_1) \Rightarrow t'_0 = 1,$$
which is still a conflict. Thus we prove this theorem.
\end{proof}

\begin{remark}
The two local non-rigidity theorems shows the rigidity will have interesting behavior at $a = a_0$, and the plus and minus have such subtle difference as above. This is because the comparison of $x_a(t)$ and $\sqrt{1-a^2}$ at $t<1$ and $t > 1$ is consistent when $a = a_0$, but is opposite when $a \neq a_0$, which is virtually, a transcritical bifurcation.
\end{remark}

Summarizing the results in Section~2, 3 and 5, we complete the proof to the main theorems in the introduction.

\section{Appendix: Calculation involving elliptical integrals}

Now we will prove Lemma~\ref{H=1}. It needs some estimations of elliptical integral.
We will calculate the derivative of the $x_a$ function to finish the estimation.
It is easy to get 
$$x_a(t) = \int_{a}^{t} \frac{1}{D(1,t, x_1)}dx_1$$
We define $$f(a,t) = x_a(t) - \sqrt{1 - a^2}.$$

It is apparent that 
$$f(a, 1) = 0, \forall a \in (0,1).$$
So we only need to consider the relationship between $f(a,t)$ and $0$ when $t \approx 1$.

Notice that $f$ is a function of elliptic integral, so we will transform it to the standard form of elliptic integral first.
\begin{align*}
f(a,t) &= \int_{a}^{t} \frac{1}{\sqrt{(\frac{x_1}{x_1^2 - t^2 +t})^2 - 1}} dx_1 - \sqrt{1 - a^2}\\
&= \int_{a}^{t} \frac{x_1^2-t^2+t}{\sqrt{(t^2 - x_1^2)(x_1^2 - (1-t)^2)}} dx_1 - \sqrt{1 - a^2}\\
&\overset{x_1 = ut}{=} \int_{\frac{a}{t}}^{1} \frac{(u^2 t^2-t^2+t)t}{\sqrt{t^2(1 - u^2)(u^2 t^2 - (1-t)^2)}} du - \sqrt{1 - a^2}\\
&= \int_{\frac{a}{t}}^{1} \frac{u^2 t-t+1}{\sqrt{(1 - u^2)(u^2 - (\frac{1-t}{t})^2)}} du - \sqrt{1 - a^2}\\
&\overset{u = \cos \theta}{=} \int^{\arccos \frac{a}{t}}_{0} \frac{1-t \sin^2 \theta}{\sqrt{\cos ^2 \theta - (\frac{1-t}{t})^2}} d\theta - \sqrt{1-a^2}.
\end{align*}
Define 
$$k(t) = \frac{t}{\sqrt{2t-1}}, \theta(a, t) = \arccos \frac{a}{t}.$$
Then we have $$f(a,t) = \int^{\theta}_{0} \frac{k-kt \sin^2 \theta}{\sqrt{1-k^2 \sin ^2 \theta}} d\theta - \sqrt{1-a^2}.$$
Define two kinds of the elliptic integral as:
\begin{align*}
& F(k, \theta) = \int_{0}^{\theta} \frac{d\phi}{\sqrt{1-k^2 \sin^2 \phi}}\\
& E(k, \theta) = \int_{0}^{\theta} \sqrt{1-k^2 \sin^2 \phi} d\phi.
\end{align*}
Then we have
\begin{align*}
f(a,t) &= \frac{1-t}{\sqrt{2t-1}}F(k, \theta)+\sqrt{2t-1} E(k, \theta) - \sqrt{1 - a^2}\\
&= (k - \frac{t}{k})F(k, \theta)+\frac{t}{k}E(k, \theta) - \sqrt{1 - a^2}.
\end{align*}
It is known that
\begin{align*}
&\frac{\partial F}{\partial k} = \frac{E(k,\theta)}{k(1-k^2)} - \frac{F(k,\theta)}{k} - \frac{k\sin 2\theta}{2(1-k^2)\sqrt{1-k^2\sin^2 \theta}}\\
&\frac{\partial E}{\partial k} = E(k,\theta) - F(k, \theta).
\end{align*}
So we can calculate the differential of $f$:
\begin{align*}
\frac{\partial f}{\partial t} &= \frac{-t}{(2t-1)^{\frac{3}{2}}}F(k,\theta)+ \frac{1-t}{\sqrt{2t-1}}[\frac{\partial F}{\partial k} \frac{t-1}{(2t-1)^{\frac{3}{2}}}+\frac{\partial F}{\partial \theta} \frac{a}{t \sqrt{t^2 - a^2}}]\\
&+ \frac{1}{\sqrt{2t-1}} E(k,\theta) + \sqrt{2t-1} [\frac{\partial E}{\partial k} \frac{t-1}{(2t-1)^{\frac{3}{2}}}+\frac{\partial E}{\partial \theta} \frac{a}{t \sqrt{t^2 - a^2}}]\\
&= \frac{-t}{(2t-1)^{\frac{3}{2}}}F(k,\theta) + \frac{1}{\sqrt{2t-1}} E(k,\theta) - \frac{(t-1)^2}{(2t-1)^2} \frac{\partial F}{\partial k} + \frac{t-1}{2t-1} \frac{\partial E}{\partial k}\\
&+\frac{a}{t \sqrt{t^2 - a^2}}[\frac{1-t}{\sqrt{2t-1}} \frac{\partial F}{\partial \theta} + \sqrt{2t-1} \frac{\partial E}{\partial \theta}]\\
&= \frac{-t}{(2t-1)^{\frac{3}{2}}}F(k,\theta) + \frac{1}{\sqrt{2t-1}} E(k,\theta)+ \frac{t-1}{2t-1} (E(k,\theta) - F(k, \theta))\\
& + \frac{(t-1)^2}{(2t-1)^2} \frac{F(k,\theta)}{k} +\frac{1}{2t-1}[\frac{E(k,\theta)}{k} - \frac{k\sin 2\theta}{2\sqrt{1-k^2\sin^2 \theta}}]\\
&+\frac{a}{t \sqrt{t^2 - a^2}}[\frac{1-t}{\sqrt{(2t-1)(1-k^2\sin^2 \theta)}} + \sqrt{(2t-1)(1-k^2\sin^2 \theta)}].
\end{align*}
Define $\Delta(k, \theta) = \sqrt{1-k^2\sin^2 \theta}$, and we can then get
\begin{align*}
\frac{\partial f}{\partial t} &= -F(k,\theta)[\frac{t-1}{2t-1}+\frac{1}{t\sqrt{2t-1}}] + E(k,\theta)[\frac{t-1}{2t-1}+\frac{t+1}{t\sqrt{2t-1}}]\\
&+\frac{a}{t \sqrt{t^2 - a^2}}[\frac{1-t}{\sqrt{2t-1} \Delta(k, \theta)} + \sqrt{2t-1} \Delta(k, \theta)]- \frac{k\sin 2\theta}{2(2t-1)\Delta(k, \theta)}.
\end{align*}
What we care about is the value when $t = 1$, and since 
$$k(1) = 1 \quad \theta(a,1) = \arccos a \triangleq \theta_1 \quad \Delta(1, \theta_1) = \cos \theta_1.$$

We compute out that
\begin{align*}
\frac{\partial f}{\partial t} |_{t = 1} &= -F(1, \theta_1) + 2E(1,\theta_1) + \frac{a}{\sqrt{1-a^2}} \cos \theta_1-\sin\theta_1\\
&= -\ln \frac{1+\sqrt{1-a^2}}{a} + \sqrt{1-a^2} + \frac{a^2}{\sqrt{1 - a^2}}\\
&= -\ln \frac{1+\sqrt{1-a^2}}{a} +  \frac{1}{\sqrt{1 - a^2}} \triangleq g(a).
\end{align*}
It is easy to verify that $g(a)$ increases monotonically, and has a unique zero point, defined as $a_0$ which is roughly $a_0 \approx 0.5524$ by numeric computation.

Therefore, when $a > a_0$, we have $$\frac{\partial f}{\partial t} |_{t = 1} > 0.$$
By its continuity, we can select $0 < \eta_a \ll 1$, such that $$\frac{\partial f}{\partial t} > 0 \quad \forall t \in (1-\eta_a,1+\eta_a).$$

And since $f(a, 1) = 0$, the $(1)$ of Lemma~\ref{H=1} is apparent.

Similarly, when $a < a_0$, we have 
$$\frac{\partial f}{\partial t} |_{t = 1} < 0,$$
and we can select $0 < \eta_a \ll 1$, such that 
$$\frac{\partial f}{\partial t} < 0 \quad \forall t \in (1-\eta_a,1+\eta_a)$$
which shows that the $(2)$ of the lemma is true.

However, when $a=a_0$, we can not copy the method above, and we need to calculate the second derivative of $f(a,t)$. But since we only need $\partial^2 f / \partial t^2 |_{t=1}$, we do not need to write all the equations down.

We can easily prove 
$$\Delta(k,\theta) = \sqrt{\frac{a^2-(t-1)^2}{2t-1}}.$$
So there is
\begin{align*}
\frac{\partial ^2 f}{\partial t^2}|_{t = 1} &= -[\frac{\partial F}{\partial k} \frac{t-1}{(2t-1)^{\frac{3}{2}}}+\frac{\partial F}{\partial \theta} \frac{a}{t \sqrt{t^2 - a^2}}]|_{t=1}-F(1,\theta_1) \cdot \frac{d}{dt}|_{t=1}[\frac{t-1}{2t-1}+\frac{1}{t\sqrt{2t-1}}]\\
&+2 [\frac{\partial E}{\partial k} \frac{t-1}{(2t-1)^{\frac{3}{2}}}+\frac{\partial E}{\partial \theta} \frac{a}{t \sqrt{t^2 - a^2}}]|_{t=1}+E(1,\theta_1) \cdot \frac{d}{dt}|_{t=1}[\frac{t-1}{2t-1}+\frac{t+1}{t\sqrt{2t-1}}]\\
&+\frac{d}{dt}|_{t=1}[\frac{a}{t\sqrt{t^2-a^2}}(\frac{1-t}{\sqrt{a^2-(t-1)^2}} + \sqrt{a^2-(t-1)^2}) - \frac{a\sqrt{t^2-a^2}}{t(2t-1)\sqrt{a^2-(t-1)^2}}]\\
&= -\frac{1}{\sqrt{1-a^2}}+ F(1,\theta_1) + \frac{2a^2}{\sqrt{1-a^2}} -2\sqrt{1-a^2}+ \frac{a^4-2a^2}{(1-a^2)^\frac{3}{2}} - \frac{2}{\sqrt{1-a^2}}+3\sqrt{1-a^2}\\
&= \ln \frac{1+\sqrt{1-a^2}}{a} +\frac{a^2-2}{(1-a^2)^\frac{3}{2}} \ \triangleq h(a).
\end{align*}
It is easy to calculate that $$h(a_0) = \frac{a_0^2-2}{(1-a_0^2)^\frac{3}{2}} + \frac{1}{\sqrt{1-a_0^2}} = -\frac{1}{(1-a_0)^\frac{3}{2}} <0\Rightarrow \frac{\partial ^2 f(a_0, t)}{\partial t^2}|_{t = 1} < 0$$
And since $$\frac{\partial f(a_0, t)}{\partial t}|_{t = 1} = 0 \quad f(a_0, 1) = 0.$$

It is apparent that there exists $ 0 < \eta_{a_0} \ll 1$, such that $$f(a_0, t) < 1 \quad \forall t \in (1 - \eta_{a_0}, 1+\eta_{a_0}) \backslash \{ 1 \}.$$
This finishes the proof of Lemma~\ref{H=1}.

\end{document}